\documentclass[11pt]{amsart}

\usepackage{amsmath,amssymb,amsthm}
\usepackage{thmtools,mathtools,mathrsfs}
\usepackage{enumerate}
\usepackage[dvipsnames]{xcolor}

\usepackage[english]{babel}
\usepackage{csquotes}
\usepackage{tikz-cd}
\usetikzlibrary{babel}

\usepackage{todonotes}
\usepackage{geometry}
\usepackage{tikz-cd}

\usepackage{hyperref}
\hypersetup{
	colorlinks=true,
	linkcolor=NavyBlue,
	urlcolor=teal,
	citecolor=OliveGreen}

\theoremstyle{plain}

\newtheorem{thm}{Theorem}[section]
\newtheorem{prop}[thm]{Proposition}
\newtheorem{lem}[thm]{Lemma}
\newtheorem{cor}[thm]{Corollary}

\theoremstyle{definition}

\newtheorem{dfn}[thm]{Definition}

\newtheorem{rem}[thm]{Remark}

\newtheorem*{comm}{Comment}
\newtheorem*{not_conv}{Notation and Conventions}

\newcommand{\N}{\mathbb{N}}
\newcommand{\Z}{\mathbb{Z}}
\newcommand{\Q}{\mathbb{Q}}
\newcommand{\R}{\mathbb{R}}
\newcommand{\C}{\mathbb{C}}

\renewcommand{\P}{\mathbb{P}}
\newcommand{\OO}{\mathcal{O}}

\DeclareMathOperator{\Id}{Id}
\DeclareMathOperator{\Aut}{Aut}

\DeclareMathOperator{\Pic}{Pic}

\DeclareMathOperator{\NS}{NS}
\DeclareMathOperator{\rk}{\mathrm{rk}}

\DeclareMathOperator{\Ima}{Im}
\DeclareMathOperator{\GL}{GL}
\DeclareMathOperator{\Hom}{Hom}
\newcommand*{\sheafhom}{\mathscr{H}\kern -2pt om}
\DeclareMathOperator{\tr}{Tr}

\DeclareMathOperator{\Spec}{Spec}
\DeclareMathOperator{\Hilb}{Hilb}
\DeclareMathOperator{\Kum}{Kum}
\DeclareMathOperator{\Sym}{Sym}

\DeclareMathOperator{\Def}{Def}
\DeclareMathOperator{\Deflt}{\mathrm{Def}^{\mathrm{lt}}}

\DeclareMathOperator{\Amp}{\mathrm{Amp}}

\begin{document}

    \title[Deformations of PEVs]{Deformations of Primitive Enriques Varieties}
    
    \author[F.\ A.\ Denisi]{Francesco Antonio Denisi}
    \address{Université Paris Cité and Sorbonne Université, CNRS, IMJ-PRG, F-75013 Paris, France}
    \email{denisi@imj-prg.fr}
    
    \author[N.\ Tsakanikas]{Nikolaos Tsakanikas}
    \address{Institut de Math\'ematiques (CAG), \'Ecole Polytechnique F\'ed\'erale de Lausanne (EPFL), 1015 Lausanne, Switzerland}
    \email{nikolaos.tsakanikas@epfl.ch}
    
    \author[Z.\ Xie]{Zhixin Xie}
    \address{Institut \'Elie Cartan de Lorraine, Universit\'e de Lorraine, 54506 Nancy, France}
    \email{zhixin.xie@univ-lorraine.fr}

    \thanks{The authors would like to thank Jefferson Baudin, Pietro Beri, L\'eo Navarro Chafloque, Stefano Filipazzi, Claudio Onorati, Gianluca Pacienza, Andrea Petracci and Matei Toma for many valuable conversations related to this work, as well as for providing useful feedback on an earlier draft. FAD is supported by the European Research Council (ERC) under European Union's Horizon 2020 research and innovation programme (ERC-2020-SyG-854361-HyperK). NT is supported by the ERC starting grant $\#804334$. \newline
    }

    \begin{abstract}
        We develop the deformation theory of primitive Enriques varieties, which are defined as quasi-\'etale quotients of primitive symplectic varieties by nonsymplectic group actions. 
        In particular, we establish a local Torelli theorem for primitive Enriques varieties. As applications thereof, we describe the behavior of certain primitive Enriques varieties under locally trivial deformations.
	\end{abstract}
    
    \maketitle
    
    \begingroup
		\hypersetup{linkcolor=black}
		\setcounter{tocdepth}{1}
		\tableofcontents
	\endgroup

    \section{Introduction}

    The singular version of the celebrated Beauville--Bogomolov decomposition theorem \cite[Théorème 2]{Beauville83a} was established initially in the projective case thanks to the works of Greb--Kebekus--Peternell \cite{GKP16b}, Druel \cite{Druel18}, Druel--Guenancia \cite{DruelGuen18}, Greb--Guenancia--Kebekus \cite{GGK19}, H\"oring--Peternell \cite{HP19} and Campana \cite{Campana21}, and was later extended to the K\"ahler setting by Bakker--Guenancia--Lehn \cite[Theorem A]{BGL22}. It asserts that any compact K\"ahler variety $V$ with klt singularities and numerically trivial canonical class $K_V$ admits a finite quasi-\'etale cover $\widetilde{V} \to V$ such that $\widetilde{V}$ splits as a product 
    \[ 
        \widetilde{V} \simeq T \times \prod_i W_i \times \prod_j Z_j \, ,
    \]
    where $T$ is a complex torus, each $W_i$ is an irreducible Calabi--Yau variety, and each $Z_j$ is an irreducible symplectic variety. 

    Primitive symplectic varieties (see \autoref{dfn:PSV}) are more general than irreducible symplectic varieties (see \autoref{dfn:ISV}) and constitute singular analogs of irreducible holomorphic symplectic (IHS) manifolds by \cite[Theorem 1]{Schwald22}. This notion is also better suited for the study of symplectic varieties from the birational \cite{LehnPac16,LMP25} and the deformation-theoretic perspective \cite{BL21,BL22}: the class of primitive symplectic varieties is stable under the operations of the Minimal Model Program (MMP), while their deformation and moduli theory is well-understood and resembles the one in the smooth case. In particular, small locally trivial deformations of ($\Q$-factorial) primitive symplectic varieties are again ($\Q$-factorial) primitive symplectic varieties, and a local Torelli theorem for primitive symplectic varieties was also established; see
    \cite[Corollary 4.11, Proposition 5.5 and Lemma 5.20]{BL22}.

    From the viewpoint of the classification theory of complex varieties, it is natural to examine not only the irreducible factors appearing in the Beauville--Bogomolov decomposition theorem but also their quotients. For instance, in the smooth setting, an \emph{Enriques manifold} \cite{OS11a} is a projective manifold $Y$ which is not simply connected and whose universal covering $X$ is a projective IHS manifold. Consequently, $Y$ can be viewed as the quotient of $X$ by the finite cyclic group $\pi_1(Y) \hookrightarrow \Aut(X)$ generated by a purely nonsymplectic automorphism and acting freely on $X$. Recall also that $Y$ has torsion canonical bundle whose order coincides with the order of $\pi_1(Y)$, which is called the \emph{index} of $Y$.
    
    Enriques manifolds were defined independently and simultaneously by Oguiso--Schr\"oer \cite{OS11a} and by Boissière--Nieper-Wi{\ss}kirchen--Sarti \cite{BNWS11}; for a comparison of their definitions see the paragraph after \cite[Definition 4.2]{BCS24}. They are regarded as higher-dimensional analogs of Enriques surfaces in the same way that IHS manifolds are viewed as higher-dimensional generalizations of K3 surfaces. All currently known examples of Enriques manifolds have index $2$, $3$ or $4$ and were constructed in op.\ cit.\ using punctual Hilbert schemes, generalized Kummer varieties and moduli spaces of sheaves on certain K3 surfaces.
    Their deformation theory was developed in \cite{OS11b}, where a local Torelli theorem for Enriques manifolds \cite[Theorem 2.4]{OS11b} was obtained.
    
    \medskip
    
    In our earlier joint work \cite{DOTX24} with Ríos Ortiz,
    we introduced the notion of primitive Enriques varieties as a singular version of Enriques manifolds: they are finite quasi-\'etale quotients of primitive symplectic varieties by nonsymplectic group actions (see \autoref{dfn:PEV}). In particular, we studied their birational geometry from the perspective of the MMP and we established a termination statement for \emph{Enriques pairs} (that is, log canonical pairs whose underlying variety is an Enriques manifold); see \cite[Theorem 1.2]{DOTX24}. Our result demonstrates in particular that primitive Enriques varieties naturally appear as underlying varieties of minimal models of Enriques pairs.
    
    In the present paper, which is an extension of \cite{DOTX24}, we develop the deformation theory of primitive Enriques varieties, building on the works \cite{OS11b,Graf18,BL22}. Even though our definition \cite[Definition 5.10]{DOTX24} was motivated by the aforementioned MMP considerations, the defor\-mation-theoretic results obtained here provide further evidence that the definition of a primitive Enriques variety as a singular counterpart of an Enriques manifold is indeed an appropriate one. 
    Before stating our main results, we recall below some relevant notation for the reader's convenience, referring to Section \ref{section:preliminaries} for more details. 

    Let $X$ be a primitive symplectic variety. According to \cite[Theorem 4.7 and Corollary 4.11]{BL22}, the space $\Deflt(X)$ of locally trivial deformations of $X$ is smooth of dimension $h^1(X,\mathscr{T}_X)= h^{1,1}(X)$ and the locally trivial Kuranishi family $\mathscr{X}\to \Deflt(X)$ of $X$ is universal for locally trivial deformations for all of its fibers.
    Now, let $Y\simeq X/G$ be a primitive Enriques variety, where $G \subseteq \Aut(X)$ is generated by a purely nonsymplectic automorphism $g$ of finite order $d \geq 2$ and acts freely in codimension one on $X$. We denote by $\big( H^2(X,\Z)_{\mathrm{tf}} \, , q_X \big)$ the Beauville--Bogomolov--Fujiki lattice of $X$ and by $H^2(X,\C)_1$ the subspace of $H^2(X,\C)$ on which $g$ acts via multiplication by a primitive $d$-th root of unity. In particular, we have
    $H^{2,0}(X) \hookrightarrow H^2(X,\C)_1$, 
    and the Hodge decomposition of $H^2(X,\C)$ yields
    \[
        H^2(X,\C)_1 \simeq H^{2,0}(X)_1 \oplus H^{1,1}(X)_1 \oplus H^{0,2}(X)_1 ,
    \]
    where $H^{1,1}(X) \simeq H^1 \big( X, \Omega^{[1]}_X \big)$.
    
    Our first main result is the following generalization of \cite[Proposition 1.2]{OS11b} to the singular setting.

    \begin{thm}[The locally trivial Kuranishi family of a PEV]
        \label{thm:locally_trivial_Kuranishi_family_of_PEV}
        Let $Y\simeq X/G$ be a primitive Enriques variety, where $X$ is a primitive symplectic variety and $G \subseteq \Aut(X)$ is generated by a purely nonsymplectic automorphism $g \in \Aut(X)$ of finite order $d \geq 2$ and acts freely in codimension one on $X$. The following statements hold:
        \begin{enumerate}[\normalfont (i)]
            \item Let $B$ be the connected component of the fixed locus of the induced action of $G$ on $\Deflt(X)$ that contains the point $0\in \Deflt(X)$. The space
            $\Deflt(Y)\simeq B$ of locally trivial deformations of $Y$ is smooth of dimension 
            \[ 
                h^1(Y, \mathscr{T}_Y) = \dim_\C H^1(X,\mathscr{T}_X)^G = \dim_\C H^1 \big( X, \Omega_X^{[1]} \big)_1 = h^{1,1}(X)_1 \, , 
            \]
            and the locally trivial Kuranishi family 
            \[
                \mathscr{Y}\coloneqq \big( \mathscr{X} \times_{\Deflt(X)} \Deflt(Y) \big)/G  \to \Deflt(Y) 
            \]
            of $Y$ is universal for locally trivial deformations for all of its fibers.
            
            \item If $S$ is a sufficiently small neighborhood of $0 \in \Deflt(X)$, then for each $b \in S\cap B$ the fiber $\mathscr{Y}_b$ is a primitive Enriques variety isomorphic to the quotient $\mathscr{X}_b/G$ of the primitive symplectic variety $\mathscr{X}_b$. In particular, both $\mathscr{X}_b$ and $\mathscr{Y}_b$ are projective for every $b \in S \cap B$.

            \item Assume, finally, that the central fiber $Y \simeq \mathscr{Y}_0$ is $\Q$-factorial. If $T$ is a sufficiently small neighborhood of $0 \in \Deflt(Y)$, then each fiber $\mathscr{Y}_t \simeq \mathscr{X}_t / G$, $t \in T$, is also $\Q$-factorial.
        \end{enumerate}
    \end{thm}

    An immediate consequence of the above theorem and the definition of universality of a deformation is the following analog of \cite[Corollary 4.11 and Lemma 5.20]{BL22} for primitive Enriques varieties: small locally trivial deformations of ($\Q$-factorial) primitive Enriques varieties are again ($\Q$-factorial) primitive Enriques varieties.

    \medskip
    
    Our second main result, which extends \cite[Theorem 2.4]{OS11b} to the singular setting and constitutes an analog of \cite[Proposition 5.5]{BL22} for primitive Enriques varieties, is the following one.
    
    \begin{thm}[Local Torelli theorem for PEVs]
        \label{thm:local_Torelli_PEV}
        Let $Y\simeq X/G$ be a primitive Enriques variety and let  
        \[
            \Omega(Y) \coloneqq \big\{ [\omega] \in \P \big( H^2(X,\C)_1 \big) \mid q_X(\omega) = 0, \ q_X(\omega,\overline{\omega}) > 0 \big \} 
        \]
        be the period domain for $Y$ inside $\P \big( H^2(X,\C)_1 \big)$. If $ \mathscr{Y} \to \Deflt(Y)$ is the universal locally trivial deformation of $Y$ 
        and if $\mathscr{X}_t$ denotes the fiber over $t \in \Deflt(X)$ of the universal locally trivial deformation $\mathscr{X} \to \Deflt(X)$ of the primitive symplectic variety $X$, 
        then the local period map
        \[
            \mathscr{P}_Y \colon \Deflt(Y) \to \Omega(Y), \ s \mapsto \P \big( H^{2,0}(\mathscr{X}_s) \big)
        \]
        is a local isomorphism at $0 \in \Deflt(Y)$.
    \end{thm}

    The proofs of Theorems \ref{thm:locally_trivial_Kuranishi_family_of_PEV} and \ref{thm:local_Torelli_PEV} are given in Sections \ref{section:locally_trivial_deform} and \ref{section:LocalTorelliThms}, respectively.
    Proceeding along the lines of the proof of \autoref{thm:local_Torelli_PEV}, we also give a local Torelli theorem for bielliptic surfaces (see \autoref{thm:local_Torelli_bielliptic_surface}), which is crucial for the proof of \autoref{prop:deform_PEV_from_Kum}.
    This statement is implicit in the proof of \cite[Proposition 3.5]{OS11b} and should be known to the experts. However, since we were unable to find explicit references in the literature, we give a detailed presentation in Subsection \ref{subsection:localTorelli_BS}.

    \medskip
    
    In Section \ref{section:applications} we also provide some concrete applications of \autoref{thm:local_Torelli_PEV}, which are motivated by the results in \cite[Section 3]{OS11b}; namely, in Propositions \ref{prop:deform_PEV_from_Hilb}, \ref{prop:deform_PEV_from_Kum} and \ref{prop:deform_PEV_from_Moduli}, respectively, we demonstrate that small locally trivial deformations of a primitive Enriques variety of the form
    \begin{align*}
        \Hilb^n(S) / G &, \ \text{ where $S$ is a K3 surface covering a primitive Enriques surface} , 
        \\
        \Kum^n(A) / G &, \ \text{ where $A$ is the canonical abelian covering of a bielliptic surface} , \text{ or } 
        \\
        M_S(v,H) / G &, \ \text{ where $S$ is a K3 surface covering a very general Enriques surface} , \\
        & \hspace{12.5pt} v \text{ is a positive Mukai vector and $H$ is a $v$-generic polarization on $S$} ,
    \end{align*}
    are again of the same form, under certain hypotheses for both the surfaces involved and the finite group $G$ of automorphisms acting on them.
    We emphasize that $M_S(v,H)$ is, in general, a singular, projective, irreducible symplectic variety, since we do not necessarily assume $v$ to be primitive; see Subsection \ref{subsection:ISVs}.

    \medskip
    
    We conclude the introduction by drawing attention to further recent progress on the study of primitive Enriques varieties.
    First of all, a global Torelli theorem for Enriques manifolds was established in \cite[Theorem B]{Raman25}. Regarding now the construction of more examples of Enriques varieties, \cite[Corollary 3]{BGGG25} shows that Enriques manifolds cannot arise as quotients of IHS manifolds of OG10 type, whereas \cite[Section 6]{BCS24} discusses some singular examples. Finally, as a class of varieties with numerically trivial canonical divisor, the Kawamata--Morrison cone conjecture has also been investigated for primitive Enriques varieties. It was originally resolved for Enriques manifolds of prime index by \cite[Theorem 1.2]{PS23}, while more recently, thanks to a general descent result for finite covers \cite[Theorem 1.4]{Gachet25}, the cone conjecture is now settled for primitive Enriques varieties admitting a finite cover by a primitive symplectic variety with $b_2 \geq 5$.

    \section{Preliminaries}
    \label{section:preliminaries}

    \begin{not_conv}~
        \begin{enumerate}[(1)]
            \item The term \emph{variety} means an integral separated scheme of finite type over $\C$ in the algebraic setting or a reduced, irreducible, Hausdorff complex space in the analytic setting.

            \item A \emph{resolution of singularities} of a variety $V$ is a proper, surjective, bimeromorphic morphism $W \to V$ from a nonsingular variety $W$.
            
            \item Let $V$ be a normal variety. We denote by $\Omega^1_V$ the sheaf of K\"ahler differentials on $V$ and by $V_\text{reg}$ the regular locus of $V$. We consider the natural inclusion $\iota \colon V_\text{reg} \hookrightarrow V$, and we set
            \[ 
                \mathscr{T}_V \coloneqq \big( \Omega_V^1 \big)^{*} \simeq \iota_* \mathscr{T}_{V_\text{reg}}
            \]
            and  
            \[ 
                \Omega_V^{[p]} \coloneqq \big( \Lambda^p \Omega_V^1 \big)^{**} = \big( \Omega_V^p \big)^{**} \simeq \iota_* \Omega_{V_\text{reg}}^p \, \text{ for any } \, p \in \{ 0, \dots, \dim V \} ,
            \]
            where $\Omega_V^{[0]} = \OO_V$.
            We refer to the global sections of $\mathscr{T}_V$ and $\Omega^{[p]}_V$ as \emph{reflexive vector fields} and \emph{reflexive $p$-forms} on $V$, respectively.
        \end{enumerate}
    \end{not_conv}

    \begin{dfn}~
        \label{dfn:nondegenerate_lattice}
        \begin{enumerate}[(1)]
            \item A \emph{lattice} $(L,q)$ is a free $\Z$-module $L$ of finite rank endowed with a quadratic form $q \colon L \to \Z$. When $q$ is nondegenerate, 
            we say that $(L,q)$ is a \emph{nondegenerate lattice}.
            
            \item If $(L,q)$ is a nondegenerate lattice, then the $\R$-linear extension of $q(\cdot, \cdot)$ to the $\R$-vector space $L_\R \coloneqq L \otimes_\Z \R$ can be diagonalized with only $+1$ and $-1$ on the diagonal, and the \emph{signature} of $(L,q)$ is defined to be $(n_+, n_-)$, where $n_\pm$ is the number of $\pm 1$ on the diagonal; see \cite[Chapter 14]{Huy16book}. We adopt the notation $(k,*)$ with $k\in \Z_{\geq 0}$ and we say that the nondegenerate lattice $(L,q)$ has \emph{signature of type $(k,*)$} if its signature is $(k,\ell)$ for some $\ell \in \Z_{\geq 0}$.

            \item A \emph{lattice isomorphism} (or an  \emph{isometry}) between two lattices $(L,q)$ and $(L',q')$ is a $\Z$-module isomorphism $\phi \colon L \to L'$ such that $q' \big( \phi(x) \big) = q(x)$ for all $x \in L$.

            \item If a group $G$ acts on a lattice $(L,q)$,
            then we say that the $G$-action is \emph{orthogonal} if the quadratic form $q$ is $G$-invariant, i.e., $q( g \cdot \ell) = q (\ell)$ for all $g \in G$ and all $\ell \in L$.
        \end{enumerate}
    \end{dfn}
    
    \begin{dfn}
        \label{dfn:orthogonal_representation}
        Let $G$ be a finite group.
        \begin{enumerate}[(1)]
            \item A finite-dimensional complex \emph{representation} of $G$ is the data of a finite-dimensional $\C$-vector space $V$ and a group homomorphism $\rho \colon G \to \GL(V)$.

            \item Let $\rho \colon G \to \GL(V)$ be a finite-dimensional complex representation of $G$.
            Assume that $V$ is equipped with a nondegenerate quadratic form $q$.
            We say that $\rho$ is an \emph{orthogonal representation} of $G$ if its image is contained in the \emph{orthogonal group} $\mathrm{O}(V,q)$ of $(V,q)$, where
            \[
                \mathrm{O}(V,q) 
                \coloneqq \left\{ T \in \GL(V) \mid q \big( T(v) \big) = q(v) \, \text{ for all } \, v \in V \right\} .
            \]
        \end{enumerate}
    \end{dfn}

    \subsection{Hodge duality for K\"ahler spaces}
    
    For the definition and basic properties of \emph{K\"ahler} complex spaces we refer to \cite{Varouchas89,BL21,BL22} and the references therein.
    The next result generalizes \cite[Proposition 6.9]{GKP16b}. It should be well known to the experts, but we include it below nonetheless due to the lack of an explicit reference.
    
    \begin{prop}[Hodge duality]
        \label{prop:Hodge_duality}
        If $V$ is a normal, reduced, compact, K\"ahler, complex space of dimension $n$ with weakly rational singularities (e.g., klt singularities, see \cite[\S 1.4]{KS21} and \cite[Theorem 5.22]{KM98}), then for any integer $p \in \{ 0, \dots, n \}$ there are $\C$-linear isomorphisms
        \[ 
            H^0 \big( V, \Omega_V^{[p]} \big) \simeq \overline{H^p \big( V,\OO_V \big)} . 
        \]
    \end{prop}

    \begin{proof}
        Let $\rho \colon W \to V$ be a resolution of singularities of $V$, where $W$ is a compact K\"ahler manifold; see \cite[Proposition 2.7]{BL22}. We may repeat verbatim the proof of \cite[Proposition 6.9]{GKP16b}, except that we now apply \cite[Corollary 1.7]{KS21} instead of \cite[Theorem 2.4]{GKP16b},
        which shows that $\rho_* \Omega^p_W \simeq \Omega^{[p]}_V$ also in this more general setting.
    \end{proof}

    \subsection{Primitive symplectic varieties}
    \label{subsection:PSVs}
    
    Recall that a normal variety $X$ of dimension at least two is called \emph{symplectic} if it is endowed with a \emph{symplectic form} $\sigma$ (i.e., a reflexive $2$-form $\sigma \in H^0 \big( X, \Omega_X^{[2]} \big)$ such that $\sigma |_{X_\text{reg}}$ is closed and everywhere nondegenerate) such that $\sigma |_{X_\text{reg}}$ extends to a holomorphic $2$-form on some (or, equivalently, any) resolution of singularities of $X$, see \cite{Beauville00}. In particular, we have $\dim X = 2n$ for some $n \geq 1$ and $K_X \sim 0$.
    For a thorough discussion of symplectic varieties we refer to \cite{BL21,BL22,DOTX24} and the numerous relevant references therein. In this paper we are mainly concerned with the following subclass of symplectic varieties.
    
    \begin{dfn}[PSV]
        \label{dfn:PSV}
        A compact, K\"ahler, symplectic variety $X$ with symplectic form $\sigma \in H^0 \big( X, \Omega_X^{[2]} \big)$ is called \emph{primitive} if $H^1(X,\OO_X) = 0$ and $H^0 \big( X, \Omega_X^{[2]} \big) = \C \cdot \sigma $.
    \end{dfn}
    
    Fix a primitive symplectic variety $X$ of dimension $2n$. Recall that $X$ has rational singularities by \cite[Theorem 3.4(1)]{BL22} and that the torsion-free part $H^2(X,\Z)_{\text{tf}} \coloneqq H^2(X,\Z) / \text{(torsion)}$ of its second cohomology group is a free $\Z$-module of finite rank by \cite[Lemma 2.1]{BL21} and carries also a pure Hodge structure of weight two by \cite[Corollary 3.5]{BL22}. Therefore, there exists a decomposition
    \begin{equation}
        \label{eq:Hodge_decomposition_PSV}
        H^2(X,\C) \simeq H^{2,0}(X) \oplus H^{1,1}(X) \oplus H^{0,2}(X) ,
    \end{equation}
    where 
    \begin{align*}
        H^{2,0}(X) &\simeq H^0 \big( X, \Omega_X^{[2]} \big) , \\
        H^{1,1}(X) &\simeq H^1 \big( X, \Omega^{[1]}_X \big) , \text{ and } \\
        \overline{H^{2,0}(X)} &\simeq H^{0,2}(X) \simeq H^2(X, \OO_X) ,
    \end{align*}
    see \cite[Corollary 2.3(1)]{BL21}.

    \subsubsection{The Beauville--Bogomolov--Fujiki form}
    
    If $\sigma \in H^{2,0}(X)$ is the cohomology class of a symplectic form on $X$ satisfying $\int_X (\sigma\overline{\sigma})^n=1$, then there exists a quadratic form on $H^2(X,\C)$, called the \emph{Beauville--Bogomolov--Fujiki (BBF) form} and defined as
    \begin{align}
        \label{eq:BBF_form_dfn}
        q_X &\colon H^2(X,\C) \to \C , \nonumber \\[1mm] 
        \alpha \mapsto q_X(\alpha) \coloneqq \frac{n}{2} \int_X (\sigma \overline{\sigma})^{n-1} \, \alpha^2 &+ (1-n) \left( \int_X \sigma^n \, \overline{\sigma}^{n-1} \, \alpha \right) \left( \int_X \sigma^{n-1} \, \overline{\sigma}^n \, \alpha \right) , 
    \end{align}
    which does not depend on the choice of $\sigma$ and whose restriction to $H^2(X,\R)$ is a real nondegenerate quadratic form with signature $\big(3,b_2(X)-3\big)$; see 
    \cite[Section 3]{Schwald20} and
    \cite[Section 5]{BL22}, which are based on \cite{Nam01b,Matsushita01,Kir15}.
    In particular, by \eqref{eq:BBF_form_dfn} we infer that
    \begin{equation}
        \label{eq:BBF_form_evaluation_1}
        q_X(\sigma) = 0 = q_X(\overline{\sigma}) , \quad q_X(\sigma+\overline{\sigma}) = 1 
    \end{equation}
    and
    \begin{align}
        \label{eq:BBF_form_evaluation_2}
        q_X(\sigma + v) = q_X(\overline{\sigma} +v) &= q_X(v) \\[1mm]
        &= \frac{n}{2} \int_X (\sigma \overline{\sigma})^{n-1} \, v^2 \quad \text{for any } \ v \in H^{1,1}(X) . \nonumber
    \end{align}
    Moreover, according to \cite[Lemma 5.7]{BL22}, we can normalize $q_X$ so that it becomes an integral quadratic form $H^2(X,\Z)\to \Z$, and hence $q_X$ is invariant under locally trivial deformations of $X$ (see \autoref{dfn:deformation}). The \emph{BBF lattice} $\big( H^2(X,\Z)_{\text{tf}} \, , q_X \big)$ of $X$ is thus nondegenerate with signature of type $(3,*)$.

    \begin{comm}
        We distinguish between the induced nondegenerate symmetric bilinear form $q_X ( \cdot , \cdot)$ on $H^2(X,\C)$ defined by 
        \begin{equation}
            \label{eq:quadratic_form}
            q_X(\alpha,\beta) \coloneqq \frac{1}{2} \big( q_X(\alpha+\beta)-q_X(\alpha) - q_X(\beta) \big) , 
        \end{equation}
        see \cite[Remark 22]{Schwald20}, and the induced Hermitian form $q_X^h$ on $H^2(X,\C)$ defined by 
        \begin{equation}
            \label{eq:hermitian_form}
            q_X^h (\alpha,\beta) \coloneqq q_X(\alpha,\overline{\beta}) ,
        \end{equation}
        see \cite[p.\ 1636]{OS11b}.
        For some useful properties of $q_X^h$ we refer to \autoref{lem:hermitian_form_PSV} below.
    \end{comm}

    \begin{rem}
        \label{rem:orthogonality_BBF_form}
        It follows from \eqref{eq:BBF_form_evaluation_1}, \eqref{eq:BBF_form_evaluation_2} and \eqref{eq:quadratic_form} that, exactly as in the smooth case, the decomposition
        \[ 
            H^2(X,\C) \simeq \left( H^{2,0}(X) \oplus H^{0,2}(X) \right) \oplus H^{1,1}(X) 
        \]
        is orthogonal with respect to $q_X(\cdot, \cdot)$.
    \end{rem}

    \subsubsection{Nonsymplectic finite group actions on PSVs}
    
    We now recall the following standard definitions.
    
    \begin{dfn}
        Let $X$ be a primitive symplectic variety and denote by $\sigma \in H^0 \big( X, \Omega_X^{[2]} \big)$ its symplectic form.
        \begin{enumerate}[(1)]
            \item An automorphism $g \in \Aut(X)$ is called \emph{nonsymplectic} if $g^{[*]} \sigma \neq \sigma$. 
              
            \item A nonsymplectic automorphism $g \in \Aut(X)$ is called \emph{purely nonsymplectic} if there exists a primitive root of unity $\xi \neq 1$ of the same order as $g$ such that $g^{[*]} \sigma = \xi \sigma$.
        \end{enumerate}
        We also say that a finite, nontrivial subgroup $G$ of $\Aut(X)$ acts \emph{nonsymplectically} on $X$ if every element $g \in G \setminus \{ \Id_X\}$ is a nonsymplectic automorphism of $X$.
    \end{dfn}
    
    Assume now that the given primitive symplectic variety $(X, \sigma)$ admits a purely nonsymplectic automorphism $g \in \Aut(X)$ of order $d \geq 2$. Set $G \coloneqq \langle g \rangle \subseteq \Aut(X)$ and let $\xi$ be a primitive $d$-th root of unity such that $g^{[*]} \sigma = \xi \sigma$. As $h^{2,0}(X) = 1$, observe that the natural map
    \[
        \rho \colon G \to \GL \left( H^0 \big( X, \Omega^{[2]}_X \big) \right) , \ g \mapsto g^{[*]}
    \]
    induces a bijection 
    \begin{equation} \label{eq:bijection_groups}
        G \to \mu_d(\C), \ g \mapsto \xi
    \end{equation}
    onto the multiplicative group of $d$-th roots of unity, see \cite[Subsection 4.2]{DOTX24}. Note also that the group of characters $\Hom \big( \mu_d(\C) , \C^* \big)$ is cyclic of order $d$ and contains a canonical generator, namely the identity character $\xi \in \mu_d(\C) \mapsto \xi \in \C^*$. We use below the identification between $\Hom \big( \mu_d(\C) , \C^* \big)$ and $\Z / d\Z$.
    
    Consider next the linear representation
    \begin{equation}
        \label{eq:representation_H2}
        G \to \GL\left( H^2(X,\C) \right) , \ g \mapsto g^*.
    \end{equation}
    The cohomology vector space $H^2(X,\C)$ is endowed with a \emph{weight decomposition} indexed by the characters $i \in \Z / d\Z$, that is,
    \begin{equation}
        \label{eq:weight_decomp}
        H^2(X,\C) = \bigoplus_{i\in \Z/ d\Z} H^2(X,\C)_i,
    \end{equation}
    where the \emph{weight space} $H^2(X,\C)_i$ is the $\C$-vector subspace of $H^2(X,\C)$ on which the generator of $G$ acts via multiplication by $\xi^i \in \C^*$. In particular, $H^2(X,\C)_0$ is the $G$-invariant subspace of $H^2(X,\C)$, and $H^{2,0}(X) \simeq \C \cdot \sigma \subseteq H^2(X,\C)_1$.
    Recall also that by \cite[Lemmas 20 and 23]{Schwald20} we have 
    \begin{equation}
        \label{eq:BBF_form_invariance}
        q_X (\alpha) = q_X \big( g^* \alpha \big) \ \text{ for any } \, \alpha \in H^2(X,\C) , 
    \end{equation}
    which shows that the representation \eqref{eq:representation_H2}
    is orthogonal. In view of the bijection \eqref{eq:bijection_groups}, the BBF lattice $\big( H^2(X,\Z)_{\text{tf}} \, , q_X \big)$ of $X$ is also endowed with an orthogonal $\mu_d(\C)$-action. 
    We can now extend \cite[Lemma 2.1]{OS11b} to the singular setting.
    
    \begin{lem}
        \label{lem:hermitian_form_PSV}
        Let $X$ be a primitive symplectic variety and let $q_X$ be its BBF form. Assume that $X$ admits a purely nonsymplectic automorphism of order $d \geq 2$ and consider a weight decomposition of $H^2(X,\C)$ as in \eqref{eq:weight_decomp} above. Then
        \[ 
            q_X^h \colon H^2(X,\C) \times H^2(X,\C)\to \C, \ (\alpha,\beta) \mapsto q_X(\alpha,\overline{\beta}) 
        \]
        is a Hermitian form on $H^2(X,\C)$ whose restriction to the weight space $H^2(X,\C)_1$ is nondegenerate and has signature of type $(2,*)$ for $d=2$ and $(1,*)$ for $d\geq 3$.
    \end{lem}
    
    \begin{proof}
        To prove the statement, we repeat verbatim the proof of \cite[Lemma 2.1]{OS11b}, taking \eqref{eq:BBF_form_evaluation_1}, \eqref{eq:quadratic_form}, \eqref{eq:hermitian_form}, \eqref{eq:BBF_form_invariance} and \cite[Proposition 25]{Schwald20} into account.
    \end{proof}

    \subsection{Irreducible symplectic varieties}
    \label{subsection:ISVs}

    We recall here the definition of a more special kind of primitive symplectic varieties and include afterwards some relevant remarks for the sake of completeness.
    
    \begin{dfn}[ISV]
        \label{dfn:ISV}
        A compact K\"ahler variety $X$ with rational singularities is called an \emph{irreducible symplectic variety} if for every quasi-\'etale cover $\widetilde{X} \to X$ of $X$, the exterior algebra of global reflexive forms on $\widetilde{X}$ is generated by a symplectic form $\widetilde{\sigma} \in H^0 \big( \widetilde{X}, \Omega_{\widetilde{X}}^{[2]} \big)$.
    \end{dfn}
    
    Note that the above definition in the K\"ahler setting, which is \cite[Definition 1.1]{BGL22}, coincides with \cite[Definition 8.16]{GKP16b}, which is originally formulated in the projective setting; see the paragraph after \cite[Definition 1.1]{BGL22} and \cite[Corollary 5.24]{KM98}.

    By definition and by \autoref{prop:Hodge_duality}, irreducible symplectic varieties are automatically primitive symplectic, but the converse does not hold: counterexamples are mentioned, for instance, in \cite[Remark 4.6(ii)]{DOTX24}. In fact, it follows readily from the Beauville--Bogomolov decomposition theorem in the singular K\"ahler setting \cite[Theorem A]{BGL22} and the K\"unneth formula that an irreducible symplectic variety is precisely a primitive symplectic variety $X$ such that any quasi-\'etale cover $\widetilde{X}$ of $X$ is again a primitive symplectic variety; see \cite[Remark 2.3]{EFGMS25}.

    \subsubsection{Examples of ISVs: moduli spaces of sheaves on K3 surfaces}

    We first recall here some terminology which is relevant for the construction of moduli spaces of sheaves on K3 surfaces. These spaces will appear exclusively in Subsection \ref{subsection:deforming_PEV_from_moduli}. For further information about everything that follows we refer, for instance, to \cite[Chapter 6]{HuyLehn10book} and \cite[Subsections 1.1 and 2.1]{PeregoRapagnetta23}.

    \medskip
    
    Fix henceforth a K3 surface $S$ and denote by $\sigma \in H^{2,0}(S)$ the (unique up to scalar) symplectic form on $S$. Consider the even integral cohomology of $S$, i.e.,
    \[
        \widetilde{H}(S,\Z) \coloneqq H^{0}(S,\Z) \oplus H^{2}(S,\Z) \oplus H^{4}(S,\Z) ,
    \]
    and endow it with a lattice structure by defining a nondegenerate symmetric bilinear form, the so-called \emph{Mukai pairing}, as follows:
    \begin{align*}
        \langle \cdot , \cdot \rangle \colon \widetilde{H}(S,\Z) \times \widetilde{H}(S,\Z) & \rightarrow \Z \\
        (r_1, \ell_1, s_1), (r_2, \ell_2, s_2) &\mapsto -r_1s_2 + \ell_1 \ell_2 -s_1r_2 \, ,
    \end{align*}
    where $r_i \in H^{0}(S,\Z) \simeq \Z$, $\ell_i \in H^{2}(S,\Z)$ and $s_i \in H^{4}(S,\Z) \simeq \Z$.
    The resulting nondegenerate lattice $\big( \widetilde{H}(S,\Z) , \langle \cdot , \cdot \rangle \big)$ is called the \emph{Mukai lattice}. An element $v = (r, \ell, s) \in \widetilde{H}(S,\Z)$ is called a \emph{Mukai vector} and satisfies
    \[
        v^2 \coloneqq \langle v,v \rangle = \ell^2 - 2rs \in 2\Z ,
    \]
    since the intersection form on $S$ is even.
    It is said to be \emph{primitive} if it is not divisible by any integer $m \geq 2$.
    
    Consider next the complexification
    $\widetilde{H}(S,\C) \coloneqq \widetilde{H}(S,\Z) \otimes_Z \C $ of $\widetilde{H}(S,\Z)$,
    endowed with the $\C$-linear extension $\langle \cdot , \cdot \rangle_\C$ of the Mukai pairing $\langle \cdot , \cdot \rangle$. The natural Hodge structure of weight two on $H^2(S,\Z)$, which is given by the Hodge decomposition 
    \[ 
        H^2(S,\C) 
        = H^{2,0}(S) \oplus H^{1,1}(S) \oplus H^{0,2}(S) ,
    \]
    can be extended to a Hodge structure of weight two on $\widetilde{H}(S,\Z)$ by setting 
    \begin{align*}
        \widetilde{H}^{2,0}(S) &\coloneqq H^{2,0}(S) \simeq \C \sigma , 
        \\
        \widetilde{H}^{0,2}(S) &\coloneqq H^{0,2}(S) \simeq \C \overline{\sigma} , \text{ and}
        \\ 
        \widetilde{H}^{1,1}(S) &\coloneqq H^0(S,\mathbb{C})\oplus H^{1,1}(S)\oplus H^4(S,\mathbb{C}).
    \end{align*}

    For any Mukai vector $v = (r, \ell, s) \in \widetilde{H}(S,\Z)$, we define its \emph{orthogonal sublattice} in $\widetilde{H}(S,\Z)$ with respect to the Mukai pairing as
    \[
        v^{\perp} \coloneqq \big\{ w \in \widetilde{H}(S,\Z) \mid \langle v,w \rangle = 0 \big\} 
    \]
    and we denote by $(v^\perp)_\C$ its complexification, which is endowed with the restriction of $\langle \cdot , \cdot \rangle_\C$ to it.
    We also observe that $v^\perp$ inherits from $\widetilde{H}(S,\Z)$ a Hodge structure of weight two in the obvious fashion. In addition, we claim that
    \begin{equation}
        \label{eq:symplectic_form_in_orthogonal_lattice}
        \text{ if } \, \ell \in \NS(S) \simeq \Pic(S) \simeq H^2(S,\Z) \cap H^{1,1}(S) , \text{ then } \, \sigma \in (v^\perp)_\C \, .
    \end{equation}
    Indeed, due to \autoref{rem:orthogonality_BBF_form} and since the integral BBF form of $S$ coincides with the intersection product on $H^2(S,\Z)$, we infer that $\sigma \in H^{2,0}(S)$ is orthogonal -- with respect to $\langle \cdot , \cdot \rangle_\C$ and under its identification with the vector $(0,\sigma,0) \in \widetilde{H}(S,\C)$ -- to $\ell \in \Pic(S)$, and hence to the given Mukai vector $v = (r, \ell, s) \in \widetilde{H}(S,\Z)$.

    \medskip
    
    Given a coherent sheaf $\mathcal{F}$ on a K3 surface $S$, we can naturally attach to it a Mukai vector by setting
    \[
         v(\mathcal{F}) \coloneqq \operatorname{ch}(\mathcal{F}) \sqrt{\operatorname{td}(S)} = \big( \rk (\mathcal{F}), \, c_1(\mathcal{F}), \, \frac{1}{2}c_1(\mathcal{F})^2 -c_2(\mathcal{F}) + \rk(\mathcal{F}) \big) \in \widetilde{H}(S,\Z) , 
    \]
    where $\frac{1}{2}c_1(\mathcal{F})^2 -c_2(\mathcal{F}) = \operatorname{ch}_2(\mathcal{F})$. By the Hirzebruch--Riemann--Roch formula, we have
    $\chi(\mathcal{F}) = \operatorname{ch}_2(\mathcal{F}) + 2\rk(\mathcal{F})$
    and hence the third component of $v(\mathcal{F})$ is often written as $\chi(\mathcal{F}) - \rk (\mathcal{F})$. 
    By construction, the associated Mukai vector $v(\mathcal{F})=(r, \ell, s)$ is of type $(1,1)$ and satisfies one of the following conditions:
    \begin{enumerate}[(1)]
      \item $r>0$;
      
      \item $r=0$ and $\ell \in H^{1,1}(S,\Z)$ is the first Chern class of a strictly effective divisor; or
      
      \item $r = \ell = 0$ and $s>0$.
    \end{enumerate}
    In view of the previous observations, a nonzero element $v = (r, \ell, s) \in \widetilde{H}(S,\Z)$ with $\ell \in \NS(S)$ satisfying $v^2 \geq 2$ and any of the above three conditions is called a \emph{positive} Mukai vector.
    
    Finally, for the definition and the comparison of the notions of \emph{$v$-generic} and \emph{$v$-general} polarizations $H \in \Amp(S)$ with respect to a positive Mukai vector $v \in \widetilde{H}(S,\Z)$, which will appear in the following paragraph as well as in Subsection \ref{subsection:deforming_PEV_from_moduli}, we refer to \cite[Subsection 2.1]{PeregoRapagnetta23}. In particular, due to \cite[Lemmas 2.9, 2.10 and 2.12]{PeregoRapagnetta23}, we will only consider $v$-generic polarizations in this paper; see also \cite[Remark 2.13]{PeregoRapagnetta23}.
    
    \medskip

    \noindent \emph{Moduli spaces of semistable sheaves on K3 surfaces}: 
    Let $S$ be a projective K3 surface. According to \cite[Theorem 1.10]{PeregoRapagnetta23}, if $v \in \widetilde{H}(S,\Z)$ is a positive Mukai vector and if $H \in \Amp(S)$ is a $v$-generic polarization, then the moduli space $M_S(v,H)$ of Gieseker $H$-semistable sheaves on $S$ with Mukai vector $v$ is a projective irreducible symplectic variety of dimension $v^2 +2 \geq 4$ whose smooth locus $M^s_S(v,H)$ is the moduli space of Gieseker $H$-stable sheaves on $S$ with Mukai vector $v$. In particular, if $v$ is primitive, then by works of Mukai and Yoshioka we know that $M^s_S(v,H) = M_S(v,H)$ is a projective IHS manifold which is deformation equivalent to the Hilbert scheme $\Hilb^n(S)$ of $n$ points on the K3 surface $S$.
    Finally, O'Grady \cite[Main Theorem]{OGrady97} and Perego--Rapagnetta \cite[Proposition 5.12]{PeregoRapagnetta24} showed that there also exists a Hodge isometry
    \begin{equation}
        \label{eq:Hodge_isometry_moduli}
        \theta \colon \big( v^\perp , \langle \cdot,\cdot\rangle |_{v^\perp} \big) \to \Big( H^2 \big( M_S(v,H) , \Z \big) , \, q_{M_S(v,H)} \Big) .
    \end{equation}
    In particular, the symplectic form on $M_S(v,H)$ arises from the symplectic form $\sigma$ on $S$; namely, $\theta_\C (\sigma) \in H^{2,0}\big( M_S(v,H) \big) $.

    \subsection{Primitive Enriques varieties}
    \label{subsection:PEVs}
    
    In this subsection we first give the definition and establish some basic properties of the main objects of study in this paper.
    
    \begin{dfn}[PEV]
        \label{dfn:PEV}
        A \emph{primitive Enriques variety} $Y$ is the quotient of a primitive symplectic variety $X$ by a finite group $G \subseteq \Aut(X)$ whose action on $X$ is nonsymplectic and free in codimension one.
    \end{dfn}

    A primitive Enriques variety $Y \simeq X/G$ is a normal, compact, K\"ahler variety with klt singularities and torsion canonical class; see \cite[Lecture 3]{Popp77book}, \cite[Section II, Corollary 3.2.1]{Varouchas89}, and \cite[Remark 2.6(i)]{DOTX24}. At first glance, \autoref{dfn:PEV} seems to be the K\"ahler version of our earlier definition of a projective primitive Enriques variety, see \cite[Definition 5.10]{DOTX24}. However, we will now show that the two definitions coincide; that is, primitive Enriques varieties and the (a priori K\"ahler) primitive symplectic varieties determining them are automatically projective.
     
    \begin{lem}
        \label{lem:PEVs_are_projective}
        Let $X$ be a primitive symplectic variety. Assume that $X$ admits a purely nonsymplectic automorphism $g \in \Aut(X)$ of finite order such that the group $G \coloneqq \langle g \rangle$ acts freely in codimension one on $X$. Then $X$ and $Y \coloneqq X / G$ are projective.
    \end{lem}

    \begin{proof}
        Recall that $Y \simeq X/G$ is a normal, compact, K\"ahler variety with klt, and hence rational, singularities. By \cite[Proposition 5.13(iii)]{DOTX24} and \autoref{prop:Hodge_duality} we obtain
        \[ 
            h^2(Y, \OO_Y) = h^0 \big(Y, \Omega_Y^{[2]} \big) = 0 .
        \]
        Therefore, \cite[Proposition 4.10]{Graf18} implies now that $Y$ is projective, so it admits an ample line bundle $\mathcal{L}$.
        Observe that the quotient map $\gamma \colon X \to Y$ is finite. Thus, using the Finite Mapping Theorem \cite[pp.\ 64-65]{GrauertRemmert84book}, the projection formula \cite[Chapter III, Exercise 8.3]{Har77book} and the relative ampleness criterion \cite[p.\ 25, (1) $\iff$ (4)]{Nak04} for the line bundle $\gamma^* \mathcal{L}$ and the composite map $X \to Y \to \Spec (\C)$, we deduce that $\gamma^* \mathcal{L}$ is an ample line bundle on $X$, so $X$ is also projective, as claimed.
    \end{proof}
    
    \begin{rem}~
        \begin{enumerate}[(1)]
            \item The projectivity assumption that is included tacitly in the statement of \cite[Proposition 5.13(iii)]{DOTX24} does not play any role in its proof, so this result can indeed be applied in the K\"ahler setting of \autoref{lem:PEVs_are_projective}.

            \item If $\dim X \geq 4$, then the hypothesis in \autoref{lem:PEVs_are_projective} that $G$ acts freely in codimension one on $X$ is automatically satisfied by \cite[Proposition 4.13]{DOTX24}.
        \end{enumerate}
    \end{rem}

    \subsubsection{Markings and period domains for PEVs}
    
    We now define a notion of marking for primitive Enriques varieties. To this end, let $(L,q)$ be an abstract nondegenerate lattice with signature of type $(3,*)$, endowed with an orthogonal $\mu_d(\C)$-action.
    We further impose the condition that the induced Hermitian form on the weight space $L_{\C,1}\coloneqq (L\otimes_{\Z}\C)_1$ is nondegenerate with signature of type $(2,*)$ for $d=2$ and $(1,*)$ for $d\geq 3$. 
    
    Consider a primitive Enriques variety $Y \simeq X/G$, where $X$ is a primitive symplectic variety and $G \subseteq \Aut(X)$ is a finite group of order $d \geq 2$, identified with $\mu_d(\C)$, which acts nonsymplectically and freely in codimension one on $X$. In view of \autoref{lem:hermitian_form_PSV} and the fact that the BBF lattice $\big( H^2(X,\Z)_{\text{tf}} \, , q_X \big)$ of $X$ is endowed with an orthogonal $\mu_d(\C)$-action by \eqref{eq:BBF_form_invariance}, an $L$-\emph{marking} of $Y$ is a $\mu_d(\C)$-equivariant lattice isomorphism 
    \[ 
        \phi \colon \big( H^2(X,\Z)_{\mathrm{tf}} \, , q_X \big) \to (L, q) ,
    \]
    and the \emph{$L$-marked primitive Enriques variety} is denoted by $(Y,\phi)$.
    
    \begin{dfn}
        \label{dfn:period_data}
        In the above setting, we also give the following definitions:
        \begin{enumerate}[(1)]
            \item The \emph{period domain} for $L$-marked primitive Enriques varieties is defined to be
            \[
                \mathcal{D}_L \coloneqq \big\{ [\omega]\in \P(L_{\C,1}) \mid q(\omega) = 0, \ q(\omega, \overline{\omega}) > 0 \big \} .
            \]
            
            \item The \emph{period point} of an $L$-marked primitive Enriques variety $(Y,\phi)$ is defined to be the point $[\phi(\sigma)] \in \mathcal{D}_L$.
        \end{enumerate}
    \end{dfn}
    
    We comment briefly on the previous definition. In the first item above, $\overline{\omega}$ denotes the complex conjugate of $\omega$ inside the complexification $L_\C = L \otimes_\Z \C$. In the second item above, slightly abusing notation, $\phi$ also denotes the complexification of the given lattice isomorphism and $\sigma \in H^{2,0}(X) \subseteq H^2(X,\C)_1$ is the cohomology class of the (unique up to scalar) symplectic form on $X$. Since $q_X(\sigma) = 0$ and $q_X(\sigma,\overline{\sigma}) > 0$ by \eqref{eq:BBF_form_evaluation_1} and \eqref{eq:quadratic_form}, we indeed have that $[\phi(\sigma)] \in \mathcal{D}_L$, as asserted.
    
    Finally, note that the period domains $\mathcal{D}_L$ inside $\P(L_{\C,1})$ are locally closed with respect to the classical topology, and hence inherit the structure of a complex manifold. In fact, their precise description as \emph{bounded symmetric domains} is given in \cite[Proposition 2.2]{OS11b}, which remains valid also in our setting with an identical proof. In view of this and \autoref{prop:deform_PEV_from_Kum} below, see also \cite[Remark 3.6]{OS11b}.

    \subsection{Notions from deformation theory}

    We refer to \cite[Section 4]{BL22} for the basic notions of the deformation theory of compact complex spaces, see also \cite{Grauert74,FK87,Sernesi06book,Gri18}. Here, we only recall the following definitions.

    \begin{dfn}
        \label{dfn:deformation}
        A \emph{deformation} $\mathscr{D}$ of a compact complex space $X$ is given by a flat and proper morphism $\pi\colon \mathscr{X}\to S$ of complex spaces together with a distinguished point $0\in S$ and an isomorphism $\alpha \colon \mathscr{X}_0 \coloneqq \pi^{-1}(0) \to X$.
        
        We say that $\mathscr{D} = (\pi \colon \mathscr{X} \to S, \, \alpha\colon \mathscr{X}_0\to X)$ is \emph{locally trivial at $0 \in S$} if for every $p \in \mathscr{X}_0 \simeq X$ there exist open neighborhoods $\mathscr{U}$ of $p$ in $\mathscr{X}$ and $V$ of $0$ in $S$ such that $\mathscr{U} \simeq U \times V$, where $U \coloneqq \mathscr{U} \cap \mathscr{X}_0$. Finally, we say that $\mathscr{D}$ is \emph{locally trivial} if it is locally trivial at each point $s \in S$.
    \end{dfn}
    
    \begin{dfn}
        \label{dfn:notions_of_versality}
        A deformation $(\mathscr{X}\to S, \, \alpha\colon \mathscr{X}_0\to X)$ of a compact complex space $X$ is called \emph{versal} if for every deformation $(\mathscr{X}'\to S', \, \beta \colon \mathscr{X}'_0\to X)$ of $X$ there exist a map $\varphi \colon (S',0)\to (S,0)$ of germs of complex spaces and an isomorphism $\mathscr{X}'\to \mathscr{X}\times_S S'$ such that the following diagrams commute:
        \begin{center}
            \begin{tikzcd}
                \mathscr{X'} \arrow[r, "\sim"] \arrow[rd] \arrow[rr, "\psi", bend left] & \mathscr{X}\times_S S' \arrow[r] \arrow[d] & \mathscr{X} \arrow[d] \\
                 & S' \arrow[r, "\varphi"] & S ,
            \end{tikzcd}
            \qquad
            \begin{tikzcd}
                \mathscr{X}'_0 \arrow[rd, "\beta"'] \arrow[rr, "\psi_0"] & & \mathscr{X}_0 \arrow[ld, "\alpha"] \\
                & X .                         
            \end{tikzcd}
        \end{center}
        A versal deformation is called \emph{miniversal} if the differential $d\varphi_0 \colon T_0 S' \to T_0 S$ of $\varphi$ is uniquely determined, and \emph{universal} if moreover the map $\varphi$ is unique.

        The different notions of versality are defined analogously for \emph{locally trivial} deformations of $X$.
    \end{dfn}

    \begin{dfn}
        \label{dfn:automorphisms_of_deformation}
        An \emph{automorphism of a deformation} $\mathscr{D} = (\mathscr{X} \to S, \, \alpha\colon \mathscr{X}_0\to X)$ of a compact complex space $X$ is an $S$-automorphism $\psi \colon \mathscr{X} \to \mathscr{X}$ such that $\psi |_{\mathscr{X}_0} = \Id_{\mathscr{X}_0}$. We denote by $\Aut(\mathscr{D})$ the \emph{automorphism group} of $\mathscr{D}$.
    \end{dfn}
    
    Any compact complex space $X$ admits a miniversal deformation $\mathscr{X} \to \Def(X)$ by \cite[Hauptsatz, p.\ 140]{Grauert74}, called the \emph{Kuranishi family} of $X$, as well as a locally trivial miniversal deformation $\mathscr{X} \to \Deflt(X)$ by \cite[Corollary 0.3]{FK87}, called the \emph{locally trivial Kuranishi family} of $X$. Note that $\Deflt(X)$ is a closed complex subspace of $\Def(X)$ by the proof of \cite[Corollary 0.3]{FK87}.

    \subsubsection{Auxiliary results}

    The content of the following result should be well known to the experts.
    
    \begin{prop}
        \label{prop:extension_of_line_bundles}
        Let $\pi \colon \mathscr{X} \to S$ be a flat and proper morphism of complex spaces with reduced fibers $\mathscr{X}_s \coloneqq \pi^{-1}(s)$, $s \in S$, where $\mathscr{X}$ (and hence $S$) is also reduced. Assume that $\pi$ is \emph{topologically locally trivial}, i.e., the underlying continuous map is a topological fiber bundle, 
        and that $R^2 \pi_* \OO_{\mathscr{X}} = 0 $.
        Given $t \in S$ and after sufficiently shrinking $S$ around $t$, the following statements hold: 
        \begin{enumerate}[\normalfont (i)]
            \item If $L_t$ is a holomorphic line bundle on $\mathscr{X}_t$, then there exists a holomorphic line bundle $\mathcal{L}$ on $\mathscr{X}$ such that $c_1 \big( \mathcal{L} |_{\mathscr{X}_t} \big) = c_1(L_t)$.

            \item If, additionally, $R^1\pi_*\OO_{\mathscr{X}}=0$, then the lift in \emph{(i)} is unique (up to isomorphism).
        \end{enumerate}
    \end{prop}
    
    \begin{proof}~

        \medskip

        \noindent (i) 
        By shrinking $S$, we can assume that it is Stein. First, we show that 
        \begin{equation}
            \label{eq:vanishingH2}
            H^2(\mathscr{X},\OO_{\mathscr{X}}) = 0 .
        \end{equation}
        Indeed, the Leray spectral sequence applied to the morphism $\pi$ and the sheaf $\OO_\mathscr{X}$ yields
        the following exact sequence
        \[
           H^2 \big( S,\pi_*\OO_{\mathscr{X}}) \to \mathrm{ker} \left(H^2(\mathscr{X},\OO_{\mathscr{X}}) \to H^0(S,R^2\pi_*\OO_{\mathscr{X}})\right) \to H^1(S,R^1\pi_*\OO_{\mathscr{X}}).
        \]
        Since the analytic sheaves $R^i\pi_* \OO_{\mathscr{X}}$, $i \geq 0$, are coherent by the Direct Image Theorem \cite[p.\ 207, Theorem]{GrauertRemmert84book} and since the space $S$ is Stein, we have 
        \[
            H^1(S, R^1\pi_*\OO_{\mathscr{X}}) = 0 = H^2(S,\pi_*\OO_{\mathscr{X}}) .
        \]
        Thus, \eqref{eq:vanishingH2} follows immediately from the above and the hypothesis that $R^2 \pi_* \OO_{\mathscr{X}} = 0 $.

        Next, we consider the exponential sequence of the reduced complex space $\mathscr{X}$,
        which gives, together with \eqref{eq:vanishingH2}, the following exact sequence in cohomology:
        \begin{equation}
            \label{eq:expsequence}
            H^1(\mathscr{X}, \OO_{\mathscr{X}}) \to H^1(\mathscr{X}, \OO^*_{\mathscr{X}}) \xrightarrow{c_1} H^2(\mathscr{X}, \Z) \to  H^2(\mathscr{X}, \OO_{\mathscr{X}}) = 0 .
        \end{equation}
        Since $\pi \colon \mathscr{X} \to S$ is topologically locally trivial, after shrinking $S$ even further so that it becomes contractible, we infer that for any $s \in S$ there is an isomorphism
        \[
            H^2(\mathscr{X}, \Z) \simeq H^2(\mathscr{X}_s,\Z) .
        \]
        Therefore, the class $c_1(L_t) \in H^2 ( \mathscr{X}_t, \Z)$ of the given holomorphic line bundle $L_t$ on $\mathscr{X}_t$ lifts to a class $\alpha \in H^2(\mathscr{X}, \Z)$ on the total space $\mathscr{X}$, which in turn satisfies $\alpha=c_1(\mathcal{L})$ for some holomorphic line bundle $\mathcal{L}$ on $\mathscr{X}$ due to \eqref{eq:expsequence}.
        
        \medskip

        \noindent (ii)
        Since $R^1\pi_*\OO_{\mathscr{X}}=0$ by assumption, the five-term exact sequence induced by the Leray spectral sequence now gives $H^1(\mathscr{X},\OO_{\mathscr{X}})=0$. Hence, $H^1(\mathscr{X},\OO^*_{\mathscr{X}})\cong H^2(\mathscr{X},\Z)$ by \eqref{eq:expsequence}, which implies that any holomorphic line bundle $L_t$ on $\mathscr{X}_t$ uniquely lifts to $\mathscr{X}$, up to isomorphism. This completes the proof.
    \end{proof}

    \begin{rem}
        \label{rem:stronger_cond}
        Let $\pi \colon \mathscr{X} \to S$ be a flat and proper morphism of complex spaces. 
        If
        \[
            H^i (\mathscr{X}_s, \OO_{\mathscr{X}_s}) = 0 \, \text{ for all } \, s \in S ,
        \]
        then 
        \[ 
            R^i \pi_* \OO_{\mathscr{X}} = 0 
        \]
        by \cite[Chapter III, Corollary 3.5]{BS76}.
    \end{rem}
    
    Recall that a proper morphism $X \to Y$ between 
    complex spaces is called \emph{projective} if there exists a relatively ample line bundle on $X$ over $Y$; see \cite[Definition II.1.10]{Nak04}. We now derive the following application of \autoref{prop:extension_of_line_bundles} and \autoref{rem:stronger_cond}, which will play a key role in Subsection \ref{subsection:deforming_PEV_from_moduli}.
    
    \begin{cor}
        \label{cor:local_projectivity}
        Let $\pi \colon \mathscr{X} \to (S,0)$ be a \emph{smooth family}, i.e., a smooth and proper morphism between connected complex spaces whose fibers are compact, connected, complex manifolds. Assume that $\mathscr{X}$ (and hence $S$) 
        is reduced and that $\mathscr{X}_0 = \pi^{-1}(0)$ is K\"ahler with $H^2(\mathscr{X}_0, \OO_{\mathscr{X}_0}) = 0$. Then there exists a sufficiently small neighborhood $V$ of $0$ in $S$ such that the restricted morphism $\pi |_{\pi^{-1}(V)} \colon \pi^{-1}(V) \to V$ is projective.
    \end{cor}    
    
    \begin{proof}
        Note first that the given smooth family is topologically locally trivial by \cite[Theorem 14.5]{GHJ03book}.
        Moreover, since the central fiber $\mathscr{X}_0$ is by assumption a compact K\"ahler manifold such that $h^{0,2}(\mathscr{X}_0) = h^{2,0}(\mathscr{X}_0) = 0$, it is actually projective by Kodaira's embedding theorem, 
        so it admits an ample line bundle $H_0$. On the other hand, by the upper semi-continuity of the function $t \mapsto H^2(\mathscr{X}_s,\OO_{\mathscr{X}_s})$, see \cite[Chapter III, Theorem 4.12(i)]{BS76}, and after shrinking $S$, we infer that $H^2 (\mathscr{X}_s, \OO_{\mathscr{X}_s}) = 0$ for all $s \in S$.
        It follows now from \autoref{prop:extension_of_line_bundles} and \autoref{rem:stronger_cond} that, after possibly further shrinking $S$, $H_0$ lifts to a line bundle $\mathcal{H}$ on $\mathscr{X}$ such that $c_1 \big( \mathcal{H} |_{\mathscr{X}_0} \big) = c_1(H_0)$. But ampleness is an open condition in smooth families. Indeed, being Kähler is an open condition in such families; see \cite[Proposition 22.2(i)]{GHJ03book}. In particular, small perturbations of integral Kähler classes are represented by ample line bundles due to Kodaira's embedding theorem.
        Therefore, up to shrinking $S$, the line bundle $\mathcal{H}$ on $\mathscr{X}$ is relatively ample over $S$.
    \end{proof}

    The following observation will only be used for the special case of an Enriques surface in Subsection \ref{subsection:deforming_PEV_from_moduli},
    where a relative moduli space construction will be discussed.
    
    \begin{rem}
        \label{rem:local_projectivity_family_of_EMs}
        If $\big( \pi_Y \colon \mathscr{Y} \to \Def(Y) , \, \mathscr{Y}_0 \simeq Y \big)$ is the universal Kuranishi family of an Enriques manifold $Y$, then \autoref{cor:local_projectivity} applies to $\pi_Y$, showing that it restricts to a smooth projective morphism over some small open neighborhood of $0 \in \Def(Y)$. 
        Indeed, the morphism $\pi_Y \colon \mathscr{Y} \to \Def(Y)$ is smooth and proper by Kuranishi's theorem, see \cite[Einleitung, 1.]{Grauert74}. Since $\Def(Y)$ is smooth by \cite[Proposition 1.2]{OS11b}, $\pi_Y$ is also locally trivial in the sense of \autoref{dfn:deformation}, which implies that $\pi_Y$ is in particular topologically locally trivial and that $\mathscr{Y}$ is reduced as well.
        The remaining conditions in \autoref{cor:local_projectivity} are satisfied, after sufficiently shrinking the Kuranishi space $\Def(Y)$, due to \cite[Proposition 2.6 and Corollary 2.7]{OS11a} and \cite[Proposition 1.2]{OS11b}.
    \end{rem}

    \subsubsection{Deformations and group actions}

    Given a compact complex space $X$ endowed with a holomorphic group action $G$, we consider the $G$-equivariant deformations of $X$, namely, a usual deformation $(\mathscr{X}\to S, \mathscr{X}_0 \xrightarrow{\sim} X )$ of $X$ equipped with a holomorphic group action on $\mathscr{X}$ over $S$, which extends the given $G$-action on $X$.

    When $X$ is an algebraic variety, Rim showed in \cite[Corollary,  p.\ 225]{Rim80} the existence of an equivariant $G$-structure on the miniversal formal deformation of $X$, extending the given $G$-action on $X$, under the hypotheses that $X$ is complete or affine with only a finite number of singularities, and $G$ is a linearly reductive group.

    In the analytic setting, Cathelineau established in \cite[Théorème 1]{Cathelineau78} the existence of a miniversal $G$-equivariant deformation of a given compact complex space $X$ endowed with an action of a group $G$ satisfying further properties, see \cite[Définition, p.\ 392]{Cathelineau78}. When $X$ is in particular a compact complex manifold, Doan proved the following analog of Rim's result:
    
    \begin{thm}[= {\cite[Theorem 5.2]{Doan22}}]
        \label{thm:Doan}
        Let $\mathscr{X}\to S$ be the Kuranishi family of a compact complex manifold $X$ endowed with a holomorphic action of a complex reductive Lie group $G$. Then there exist local equivariant $G$-actions on $\mathscr{X}$ over $S$ extending the given $G$-action on $X$.
    \end{thm}

    One of the main problems of extending the given $G$-action on $X$ to the whole family $\mathscr{X}\to S$ is the following: when $X$ admits nonzero holomorphic vector fields, there is no canonical choice of such an extension. However, in the opposite case, we have a unique extension of the $G$-action to the universal deformation of $X$:

    \begin{lem}
        \label{lem:extending_group_action_to_deformation}
        Let $X$ be a compact complex space, let $G$ be a subgroup of $\Aut (X)$, and let $\mathscr{D} = (\pi\colon \mathscr{X}\to S, \, \alpha \colon \mathscr{X}_0\to X)$ be a versal deformation of $X$. 
        Assume that $\mathscr{D}$ is universal. Then there exists a group homomorphism 
        \[
            \Phi\colon G\to \Aut\big((S,0)\big), \ g \mapsto \varphi_g,
        \]
        where $\Aut\big((S,0)\big)$ denotes the group of germs of biholomorphisms that fix the marked point $0 \in S$.
        Assume moreover that $H^0(X, \mathscr{T}_X) = \{ 0 \}$.
        Then there exists a group homomorphism
        \[
            \Psi \colon G \to \Aut_S (\mathscr{X}), \ g \mapsto \psi_g,
        \]
        where $\Aut_S (\mathscr{X})$ denotes the group of $S$-automorphisms of the complex space $\mathscr{X}$.
    \end{lem}
    
    \begin{proof} 
        Let $g,h \in G$. Setting $\beta \coloneqq g\circ \alpha$ and $\gamma \coloneqq h\circ \alpha$, we obtain two deformations of $X$:
        \[
            \mathscr{D}_g = (\pi\colon\mathscr{X}\to S, \, \beta \colon \mathscr{X}_0\to X) \quad \text{ and } \quad \mathscr{D}_h = (\pi\colon\mathscr{X}\to S, \, \gamma \colon \mathscr{X}_0\to X) .
        \] 
        By the versality of $\mathscr{D}$ there exist two maps $\varphi_g, \varphi_h \colon (S,0) \to (S,0)$ and two morphisms $\psi_g, \psi_h \colon \mathscr{X} \to \mathscr{X}$ such that the following diagrams commute:
        \begin{center}
            \begin{tikzcd}[row sep = 4em, column sep = large]
                \mathscr{X}\simeq \mathscr{X}\times_S S \arrow[d, "\pi"'] \arrow[r, "\psi_h"] & \mathscr{X}\simeq \mathscr{X}\times_S S \arrow[d, "\pi"'] \arrow[r, "\psi_g"] & \mathscr{X} \arrow[d, "\pi"] \\
                (S,0) \arrow[r, "\varphi_h"] & (S,0) \arrow[r, "\varphi_g"] & (S,0)  
            \end{tikzcd}
        \end{center}
        and
        \begin{center}
            \begin{tikzcd}[row sep = 4em, column sep = 7em]
            \mathscr{X}_0 \arrow[d, "\alpha"'] \arrow[rd, "\gamma"] \arrow[r, "\psi_h|_{\mathscr{X}_0}"] & \mathscr{X}_0 \arrow[r, "\psi_g|_{\mathscr{X}_0}"] \arrow[d, "\alpha"'] \arrow[rd, "\beta"] & \mathscr{X}_0 \arrow[d, "\alpha"] \\
            X \arrow[r, "h"] & X \arrow[r, "g"] & X .       
            \end{tikzcd}
        \end{center}
        Since $\psi_g|_{\mathscr{X}_0} = \alpha^{-1}\circ g\circ \alpha$ and $\psi_h|_{\mathscr{X}_0} = \alpha^{-1}\circ h\circ \alpha$, we have
        \begin{equation}\label{eq:group_homomorphism}
        (\psi_g|_{\mathscr{X}_0}) \circ (\psi_h|_{\mathscr{X}_0}) = \alpha^{-1} \circ (g\circ h) \circ \alpha = \psi_{g\circ h}|_{\mathscr{X}_0}.
        \end{equation}
        Now, by the universality of $\mathscr{D}$, the map
        \[ 
                \Phi\colon G\to \Aut\big((S,0)\big), \ g \mapsto \varphi_g .
        \]
        is well defined and is a group homomorphism.
    
        If, additionally, $H^0(X, \mathscr{T}_X) = \{ 0 \}$, then $\Aut(\mathscr{D}) = \{ \Id_{\mathscr{D}} \}$; see, for example, \cite[Proposition A.1]{Gri18} and observe that smoothness does not play any role in the proof of this result.
        Hence, it follows from \eqref{eq:group_homomorphism} that the map $\Psi$ is well defined and is a group homomorphism.
    \end{proof}

    \section{Locally trivial deformations of PEVs}
    \label{section:locally_trivial_deform}
    
    We first derive an analog of \cite[Lemma 4.6]{BL22} for primitive Enriques varieties, which is also an extension of \cite[Proposition 1.1]{OS11b} to the singular setting.
    
    \begin{lem}
        \label{lem:low_coh_PEV}
        Let $Y\simeq X/G$ be a primitive Enriques variety, where $X$ is a primitive symplectic variety and the group $G \subseteq \Aut(X)$ is generated by a purely nonsymplectic automorphism $g \in \Aut(X)$ of finite order $d \geq 2$ and acts freely in codimension one on $X$. The following statements hold:
        \begin{enumerate}[\normalfont (i)]
            \item $H^0(Y,\mathscr{T}_Y)=0$.

            \item Every miniversal deformation of $Y$ is universal. 

            \item $ H^1(Y,\mathscr{T}_Y) \simeq  H^1(X,\mathscr{T}_X)_0 \simeq H^{1,1}(X)_1 $.
        \end{enumerate}   
    \end{lem}

    \begin{proof}
        Since $G$ acts freely in codimension one on $X$, \cite[Lemma 5.3]{Graf18} yields the isomorphisms 
        \begin{equation}
            \label{eq:coh_tang_quotient}
            H^i(Y,\mathscr{T}_Y) \simeq H^i(X,\mathscr{T}_X)^G \ \text{ for all } \ i \geq 0 .
        \end{equation}

        \noindent (i) Follows immediately from \eqref{eq:coh_tang_quotient} for $i=0$ and \cite[Lemma 4.6]{BL22}.

        \medskip

        \noindent (ii) The assertion is a consequence of (i); see \cite[Subsection 4.3]{BL22}.

        \medskip

        \noindent (iii) 
        Consider the canonical inclusion $\iota \colon X_\mathrm{reg} \hookrightarrow X$.
        Since the sheaves of $\OO_X$-modules $\mathscr{T}_X$ and $\Omega_X^{[1]}$ are reflexive, and hence $\mathscr{T}_X \simeq \iota_* \mathscr{T}_{X_{\text{reg}}}$ and $\Omega^{[1]}_X \simeq \iota_* \Omega^1_{X_{\text{reg}}}$,
        the (unique up to scalar) symplectic form $\sigma \in H^0 \big( X, \Omega_X^{[2]} \big)$ on $X$ induces
        an isomorphism 
        \[ 
            \mathscr{T}_X \to \Omega_X^{[1]} , \ V \mapsto \big( V' \mapsto \sigma(V,V') \big) \, .
        \]
        If $\xi \in \C^*$ is a primitive $d$-th root of unity such that $g^{[*]} \sigma = \xi \sigma$, then it is easy to check that a (locally defined) $G = \langle g \rangle$-invariant reflexive vector field $V$ on $X$ corresponds under the above isomorphism to a (locally defined) reflexive $1$-form $\omega_V$ on $X$ on which $g$ acts via multiplication by $\xi \in \C^*$.
        Therefore, 
        \begin{equation}
            \label{eq:first_coh_under_iso}
            H^1(X,\mathscr{T}_X)_0 \simeq H^1 \big( X, \Omega^{[1]}_X \big)_1 \, .
        \end{equation}
        We obtain (iii) by combining \eqref{eq:first_coh_under_iso} and \eqref{eq:coh_tang_quotient} for $i=1$.
    \end{proof}
    
    We now prove the first main result of the paper, \autoref{thm:locally_trivial_Kuranishi_family_of_PEV}, using \autoref{lem:low_coh_PEV} and building on \cite[Theorem 4.7 and Corollary 4.11]{BL22} and \cite[Propositions 6.1 and 6.2]{Graf18}, see also \cite[Proposition 5.4]{GS21}, avoiding though the use of \cite[Corollary, p.\ 225]{Rim80}.
    
    \begin{proof}[Proof of \autoref{thm:locally_trivial_Kuranishi_family_of_PEV}]~

        \medskip

        \noindent (i)
        To prove all the assertions, we follow the same strategy as the one for the proofs of \cite[Proposition 6.1]{Graf18} and \cite[Proposition 1.2]{OS11b}.
            
        Denote by 
        $ \mathscr{D} = \big( \pi_X \colon \mathscr{X} \to \Deflt(X) , \, \mathscr{X}_0 \simeq X \big) $
        the locally trivial Kuranishi family of the primitive symplectic variety $X$. By \cite[Theorem 4.7 and Corollary 4.11]{BL22} we know that $\Deflt(X)$ is smooth of dimension 
        $h^1(X,\mathscr{T}_X)= h^{1,1}(X)$, and that $\mathscr{D}$ is universal for locally trivial deformations (for all of its fibers).  
        Since $H^0(X, \mathscr{T}_X) = \{ 0 \}$ by \cite[Lemma 4.6]{BL22}, we even have $\Aut(\mathscr{D}) = \{ \Id_{\mathscr{D}} \}$. It follows from \autoref{lem:extending_group_action_to_deformation} that the group $G$ acts both on $\mathscr{X}$, extending the original $G$-action on $X \simeq \mathscr{X}_0$, and on $\Deflt(X)$, having $0\in \Deflt(X)$ as a fixed point of the $G$-action, in such a way that the morphism $\pi_X \colon \mathscr{X} \to \Deflt(X)$ becomes $G$-equivariant.

        Denote now by $B$ the connected component of the fixed locus of the induced action of $G$ on $\Deflt(X)$ that contains the point $0\in \Deflt(X)$ and recall that $T_0 \Deflt(X) \simeq H^1(X, \mathscr{T}_X)$. Since $G$ is finite, the action of $G$ on $\Deflt(X)$ linearizes locally around $0$, see \cite[Lemme 1]{Car54}. Hence, $B$ is smooth at $0$ and $T_0 B \simeq H^1(X,\mathscr{T}_X)^G \simeq H^1(Y,\mathscr{T}_Y)$ by \autoref{lem:low_coh_PEV}(iii). In conclusion, $B$ is smooth (as a germ) and of the claimed dimension.
        
        Consider next the induced family $\mathscr{Z} \coloneqq \mathscr{X}\times_{\Deflt(X)} B \to B$ and observe that $G$ acts fiberwise on $\mathscr{Z}$. According to \cite[Proposition 5.4]{GS21}, see also \cite[Proposition 6.2]{Graf18}, the morphism $\pi_Y \colon \mathscr{Y}\coloneqq \mathscr{Z}/G \to B$ is a locally trivial deformation of $ Y \simeq \mathscr{Y}_0 \simeq \mathscr{X}_0 / G$ and we have the following commutative diagram:
        \begin{equation}
            \label{eq:KS_maps}
            \begin{tikzcd}[column sep = large]
                H^1(Y,\mathscr{T}_Y) \arrow[r] & H^1(X,\mathscr{T}_X) \\
                T_0 B \arrow[r] \arrow[u] & T_0\Deflt(X) \arrow[u] .
            \end{tikzcd}
        \end{equation}
        Here, the vertical maps are the Kodaira--Spencer maps of the deformations under consideration and the upper horizontal map is induced by the restriction $\frac{1}{d}\tr$ of the trace map $\tr \colon \gamma_* \mathscr{T}_X \to \mathscr{T}_Y$ to the invariant subsheaf $(\gamma_* \mathscr{T}_X)^G \subset \gamma_* \mathscr{T}_X$, see \cite[Definition 5.1 and Lemma 5.3]{Graf18}, where $\gamma$ is the quotient map $X \to X/G \simeq Y$. By the universality of $\pi_X$, the associated Kodaira--Spencer map $T_0\Deflt(X)\to {H^1(X,\mathscr{T}_X)}$ is an isomorphism. Thus, the induced map on $G$-invariant subspaces, namely $T_0 B \to H^1(X,\mathscr{T}_X)^G \simeq H^1(Y,\mathscr{T}_Y)$, is also an isomorphism. Since $B$ is smooth at $0$, it follows that the locally trivial deformation $\pi_Y \colon \mathscr{Y} \to B$ of $Y \simeq \mathscr{Y}_0$ is miniversal, cf.\ \cite[Proposition 2.5.8]{Sernesi06book}, and \autoref{lem:low_coh_PEV}(ii) implies now that it is actually universal for locally trivial deformations. In particular, $B \simeq \Deflt(Y)$.
    
        Finally, since the universal deformation $\pi_Y \colon \mathscr{Y}\to B$ of $Y$ is locally trivial, the sheaf $\mathscr{T}_{\mathscr{Y}/B}$ is flat over $B$.
        By \autoref{lem:low_coh_PEV}(i) and by the upper semi-continuity of the function $t \mapsto H^0(\mathscr{Y}_t,\mathscr{T}_{\mathscr{Y}_t})$, see \cite[Chapter III, Theorem 4.12(i)]{BS76}, we infer that $H^0(\mathscr{Y}_b,\mathscr{T}_{\mathscr{Y}_b})=0$ for any $b\in B$. Hence, the very last part of the statement follows from the openness of versality; see \cite[Hauptsatz, p.\ 140]{Grauert74} and  \cite[Remark 5.8]{FK87}.

        \begin{center}
            \begin{tikzcd}[row sep = normal, column sep = 5em]
                \mathscr{Z} = \mathscr{X}\times_{\Deflt(X)}\Deflt(Y) \arrow[rd, "/G"'] \arrow[r, hook] & \mathscr{X} \arrow[rr, "\pi_X"] &  & \Deflt(X) \\
                & \mathscr{Y} \arrow[rr, "\pi_Y"] &  & \Deflt(Y) \arrow[u, hook]
            \end{tikzcd}
        \end{center}
                
        \noindent (ii) 
        Consider the universal locally trivial Kuranishi family $\big( \pi_X \colon \mathscr{X} \to \Deflt(X) , \, \mathscr{X}_0 \simeq X \big)$ of $X$. To prove the statement, we proceed in three steps.
        
        \medskip
        
        \emph{Step 1}:
        We will show that $G$ acts nonsymplectically on any fiber of $\pi_X$ over a sufficiently small neighborhood of $0\in \Deflt(X)$ inside $B \simeq \Deflt(Y)$.
        
        To this end, we first closely follow the proof of \cite[Corollary 4.11]{BL22}, extracting the following facts: Let $\mu \colon \mathscr{W} \to \mathscr{X}$ be a simultaneous resolution of $\pi_X \colon \mathscr{X} \to \Deflt(X)$ and denote by $j\colon \mathscr{X}_{\text{reg}}\to \mathscr{X}$ the inclusion of the regular locus. Note that 
        $\Omega^{[2]}_{\mathscr{X}/S} \simeq j_* \, \Omega^{2}_{\mathscr{X}_{\text{reg}}/S}$.
        By passing to a sufficiently small neighborhood $S$ of $0\in \Deflt(X)$, we have an isomorphism 
        \[
            (\pi_X \circ \mu)_* \, \Omega^{2}_{\mathscr{W}/S}\to (\pi_X \circ j)_* \, \Omega^{2}_{\mathscr{X}_{\text{reg}}/S} 
        \]
        of line bundles on $S$. In particular, there exists a relative holomorphic $2$-form $\sigma$ on $\mathscr{X}_{\text{reg}}$
        whose pullback extends to a relative holomorphic $2$-form on $\mathscr{W}$. After shrinking $S$ even further, 
        we also know that for any $s\in S$ the fiber $\mathscr{X}_s = \pi_X^{-1}(s)$ is a primitive symplectic variety whose symplectic form is the restriction $\sigma_s \coloneqq \sigma |_{\mathscr{X}_s}$ of $\sigma \in H^0 \big( \mathscr{X}, \Omega^{[2]}_{\mathscr{X}/S} \big)$ to $\mathscr{X}_s$.
        
        Next, slightly abusing notation, we still denote by $g$ the automorphism of $\mathscr{X}$ extending the original automorphism $g$ of $X$, where $G = \langle g \rangle$. Since the map $\pi_X \colon \mathscr{X}\to \Deflt(X)$ is $G$-equivariant, we have 
        \[ 
            g^* \Omega^{[p]}_{\mathscr{X}/\Deflt(X)} \simeq \Omega^{[p]}_{\mathscr{X}/\Deflt(X)} \ \text{ for any } p \in \{ 0, \dots, \dim X \} .
        \]
        In particular, since $\sigma \in H^0 \big( \mathscr{X}, \Omega^{[2]}_{\mathscr{X}/S} \big)$, we infer that $g^*\sigma \in H^0 \big( \mathscr{X}, \Omega^{[2]}_{\mathscr{X}/S} \big)$.
        Consider now the symplectic form $\sigma_0 = \sigma |_{\mathscr{X}_0}$ on the central fiber $\mathscr{X}_0 \simeq X$, let $\xi$ be a primitive $d$-th root of unity such that $g^* \sigma_0 = \xi \sigma_0$, and consider also the evaluation map
        \[
            H^0 \big( S, (\pi_X \circ j)_* \, \Omega^{2}_{\mathscr{X}_\text{reg} / S} \big) \to H^0 \big( \mathscr{X}_s, \Omega^{[2]}_{\mathscr{X}_s} \big) \simeq \C.
        \]
        Since for any $k \in \{2, \dots, d\}$ we have $(g^*\sigma - \xi^k\sigma) (0) \neq 0$, after possibly shrinking $S$, we deduce that for any $k\in\{2,\dots, d\}$ the value of the section $g^* \sigma - \xi^k \sigma$
        at any point $s\in S\cap B$ is also nonzero. Hence, $g^* \sigma_s = \xi \sigma_s$ for any point $s\in S\cap B$; in other words, $G$ acts nonsymplectically on each fiber over $S\cap B$, as claimed.
        
        \medskip

        \emph{Step 2}:
        We will show that $G$ acts freely in codimension one on any fiber of $\pi_X$ over a sufficiently small neighborhood of $0\in \Deflt(X)$ inside $B \simeq \Deflt(Y)$.
        
        To prove this claim, we adopt the same strategy as in Step 2 of the proof of \cite[Proposition 6.2]{Graf18}, referring to op.\ cit.\ for more details. Let $\Delta \subseteq S\cap B$ be a small polydisk centered at $0 \in \Deflt(X)$, where $S$ is a small neighborhood of $0$ as in Step 1, and let $b\in \Delta$. Let $\mathcal{U}_b\subseteq \mathscr{Y}$ be an open subset such that $\mathcal{U}_b\cap \mathscr{Y}_b\neq \emptyset$ and let $\mathcal{V}_b\subseteq \mathscr{Z}$ be one of the connected components of the preimage of $\mathcal{U}_b$ under $\mathscr{Z}\to \mathscr{Y}=\mathscr{Z}/G$. Then $\mathcal{V}_b \cap \mathscr{Z}_b \neq \emptyset$. Using reflexive vector fields in $\mathscr{T}_{\mathscr{Z}}(\mathcal{V}_b)$ and the local $\C$-action associated to them, we can construct as in the second paragraph in Step 2 of the proof of \cite[Proposition 6.2]{Graf18} a $G$-equivariant isomorphism such that the following diagram commutes:
        \begin{equation}
            \label{diag:free_in_codim_one}
            \begin{tikzcd}[column sep = large]
                \left(\mathcal{V}_b\cap \mathscr{Z}_0\right) \times \Delta \arrow[rr, " \cong "  ] \arrow[rd, "\text{pr}_2"'] &  & \mathcal{V}_b \arrow[ld, "\pi"] \\
                & \Delta
            \end{tikzcd}
        \end{equation}
        Here, the $G$-action on $\mathcal{V}_b$ is the one induced by that on $\mathscr{Z}$, while the $G$-action on the product $(\mathcal{V}_b\cap \mathscr{Z}_0)\times \Delta$ is given as follows: on the first factor it coincides with the $G$-action on the central fiber and on the second factor it is the identity.
        To show that the $G$-action on $\mathscr{Z}_b$ is free in codimension one for $b\in B \cap S$ close to $0$, we note that the total space $\mathscr{Z}$ of the family is covered by open subsets $\mathcal{V}_b$ satisfying the conditions in the diagram \eqref{diag:free_in_codim_one}. Using this diagram we see that the action on such a fiber $\mathscr{Z}_b$ is locally described by the action on the central fiber $X$, where it is free in codimension one by assumption. Since the deformation $\mathscr{Z}\to B$ is locally trivial by construction, the fiber $\mathscr{Z}_b$ can be covered using the open subsets $\mathcal{V}_b$, and we deduce that the action is free in codimension on the whole fiber $\mathscr{Z}_b$. 

        \medskip
        
        \emph{Step 3}: Conclusion.

        It follows from Steps 1 and 2 that if $S$ is any sufficiently small neighborhood of $0 \in \Deflt(X)$, then for each $b \in S\cap B$ the group $G$ acts nonsymplectically and freely in codimension one on the primitive symplectic variety $\mathscr{X}_b = \pi_X^{-1}(b)$, and hence the corre\-sponding quotient $\mathscr{X}_b/G = \mathscr{Y}_b = \pi_Y^{-1}(b)$ is a primitive Enriques variety, as claimed. Finally, the assertion that both $\mathscr{X}_b$ and $\mathscr{Y}_b$ are projective for every $b \in S \cap B$ follows immediately from \autoref{lem:PEVs_are_projective}.
        
        \medskip

        \noindent (iii)
        Consider the universal locally trivial Kuranishi family $\big( \pi_Y \colon \mathscr{Y} \to \Deflt(Y) , \, \mathscr{Y}_0 \simeq Y \big)$ of $Y$ and a simultaneous resolution $\mathscr{W} \to \mathscr{Y}$ thereof, which exists by \cite[Lemma 4.9]{BL22}. If $S$ is a sufficiently small neighborhood of $0 \in \Deflt(X)$ as in (ii), then for each $t \in S \cap \Deflt(Y)$ the fiber $\mathscr{Y}_t = \pi_Y^{-1}(t)$ is a primitive Enriques variety, so it is projective and has rational singularities. Therefore, by assumption and by \cite[Lemma 2.13(4)]{BL22}, the central fiber $\mathscr{Y}_0 \simeq Y$ satisfies \cite[Equality (2.4), p.\ 230]{BL22} with respect to the induced fiberwise resolution $\mu \colon \mathscr{W}_0 = W \to \mathscr{Y}_0 = Y$; namely, 
        \begin{equation}\label{eq:BL22_Qfactorial}
            \Ima \left( H^2(W, \Q) \to H^0 \big(Y, R^2 \mu_* \Q_W \big) \right) = \Ima \left( \bigoplus_i \Q [E_i^{\mu}] \to H^0 \big(Y, R^2 \mu_* \Q_W \big) \right) ,
        \end{equation}
        where the $E_i^\mu$ are the $\mu$-exceptional prime divisors on $W = \mathscr{W}_0$.
        This equality shows that the $\Q$-factoriality of a compact algebraic space with rational singularities depends only on the topology of a given resolution; 
        see \cite[Proposition 12.1.6]{KM92} and the proof of \cite[Proposition 12.1.7]{KM92}. By \cite[Proposition 5.1]{AV21} the equality \eqref{eq:BL22_Qfactorial} remains valid under small locally trivial deformations.
        Hence, taking \cite[Remark 4.10]{BL22} into account and passing to a topologically trivializing open subset $0 \in T \subseteq S \cap \Deflt(Y)$ for the locally trivial family $\mathscr{W} \to \mathscr{Y} \to \Deflt(Y)$, we deduce from \cite[Lemma 2.13]{BL22} that each fiber $\mathscr{Y}_t = \pi_Y^{-1}(t)$, $t \in T$, is $\Q$-factorial. This completes the proof.
    \end{proof}

    \section{Local Torelli theorems}
    \label{section:LocalTorelliThms}

    In this section we establish a local Torelli theorem for primitive Enriques varieties (see \autoref{thm:local_Torelli_PEV}) and one for bielliptic surfaces (see \autoref{thm:local_Torelli_bielliptic_surface}), similarly to the case of Enriques surfaces. These two results will be used crucially in Section \ref{section:applications} to determine the behaviour of certain primitive Enriques varieties under small deformations.

    \subsection{A local Torelli theorem for PEVs}
    \label{subsection:localTorelli_PEV}
        
    \begin{proof}[Proof of \autoref{thm:local_Torelli_PEV}]
        We proceed in four steps, following the strategy of the proof of \cite[Theorem 2.4]{OS11b}.
        The notion of marking of a primitive Enriques variety, which will be used in the proof below, can be found in Subsection \ref{subsection:PEVs}.
        
        \medskip 

        \emph{Step 1}: Construction of the local period map.
        
        By shrinking the smooth base of the locally trivial Kuranishi family $\pi_Y \colon \mathscr{Y}\to \Deflt(Y)$ of the given primitive Enriques variety $Y \simeq X / G$, we obtain a locally trivial deformation $f_Y \colon \mathscr{Y} \to S$ of $Y$ over a simply connected and smooth base $S \subseteq \Deflt(Y)$. We denote by $f_X \colon \mathscr{X} \times_{\Deflt(X)} S \to S $ the corresponding locally trivial deformation of $X$, where $\pi_X \colon \mathscr{X} \to \Deflt(X)$ is the locally trivial Kuranishi family of the primitive symplectic variety $X$. 
        Since $S$ is by construction a sufficiently small neighborhood of $0 \in \Deflt(X)$, for each $s \in S$ we have $\mathscr{Y}_s \simeq \mathscr{X}_s / G$, so that the BBF lattice $\big( H^2( \mathscr{X}_s ,\Z)_{\text{tf}} \, , q_{\mathscr{X}_s} \big)$ of $\mathscr{X}_s$ is endowed with an orthogonal $\mu_d(\C)$-action under the identification $G \leftrightarrow \mu_d(\C)$.

        Fix now an $L$-marking $\phi \colon \big( H^2(X,\Z)_{\text{tf}} \, , q_X \big) \to (L,q)$ of $Y$. Since the locally trivial family $f_X \colon \mathscr{X} \to S$ is topologically locally trivial by \cite[Proposition 5.1]{AV21}, the local system $R^2 (f_X)_* \Z_{\mathscr{X}}$ on $S$ is actually a constant sheaf on $S$ by \cite[Proposition 1.5]{ZS10}. Note that its stalk at a point $s \in S$ is canonically isomorphic to the singular cohomology group $H^2(\mathscr{X}_s,\Z)$, where $\mathscr{X}_s = f_X^{-1}(s)$. Hence, all these cohomology groups (for varying $s\in S$), and thus their torsion-free parts $H^2(\mathscr{X}_s,\Z)_\text{tf}$, are isomorphic. In view of the invariance of the BBF form under locally trivial deformations due to \cite[Lemma 5.7]{BL22}, it follows that the $L$-marking of $Y$ can be uniquely extended to an $L$-marking of the restricted family $f_Y \colon \mathscr{Y}\to S$; the fiberwise marking is denoted again by $\phi$ with a slight abuse of notation.
        
        Consequently, we obtain a period map
        \[
            \mathscr{P}_Y \colon S \to \Omega(Y), \ s \mapsto \big[ \phi(\sigma_{\mathscr{X}_s}) \big] ,
        \]
        which is holomorphic. Indeed, it is the restriction of the period map $\mathscr{P}_X$, which is itself holomorphic, since $(f_X)_* \, \Omega^{[2]}_{\mathscr{X}/S}$ is a holomorphic subbundle of $H^2(X,\C) \otimes \OO_S$; see the proof of 
        \cite[Proposition 5.5]{BL22}.

        \medskip
        
        \emph{Step 2}: Reduction step.
        
        Since both $S$ and $\Omega(Y)$ are smooth, to prove the statement, it suffices to show that 
        \begin{enumerate}[(a)]
            \item the differential $d (\mathscr{P}_Y)_0$ of $\mathscr{P}_Y$ at $0\in S$ is injective, and
            
            \item $\dim T_0 S = \dim T_{[\phi(\sigma)]} \, \Omega(Y)$.
        \end{enumerate}

        First, as far as item (a) is concerned, since the differential $d (\mathscr{P}_X)_0$ is bijective by \cite[Proposition 5.5]{BL22}, taking into account the horizontal inclusions and the vertical isomorphisms from the diagram \eqref{eq:KS_maps} we conclude that the differential $d (\mathscr{P}_Y)_0$ is injective, 
        as desired. Next, to prove item (b), we first note that
        \begin{equation}
            \label{eq:dim_S}
            \dim T_0 S = h^1(Y,\mathscr{T}_Y) = h^{1,1}(X)_1
        \end{equation}
        according to \autoref{lem:low_coh_PEV}(iii). The remainder of the proof is therefore devoted to the computation of $\dim T_{[\phi(\sigma)]} \, \Omega(Y)$. We proceed by distinguishing two cases for $d$ and we use the following observation: as $\phi$ induces an isomorphism $H^2(X,\C)_1\simeq L_{\C,1}$, by \cite[Proposition 10.1]{Voisin07book_I} we know that
        \begin{equation}
            \label{eq:dim_grass}
             T_{[\phi(\sigma)]} \, \P(L_{\C,1}) \simeq \Hom \big( \C \phi(\sigma) , L_{\C,1}/ \C \phi(\sigma) \big) \simeq \Hom \big( \C \sigma , H^2(X,\C)_1/ \C \sigma \big).
        \end{equation}
        
        \medskip
        
        \emph{Step 3}: Case $d\geq 3$.

        For any $\omega\in H^2(X,
        \C)_1$ we have $g^{[*]} \omega = \xi \omega$, and thus
        \[
            q_X(\omega) \overset{(\ast)}{=\!=} q_X \big(  g^{[*]} \omega \big) = q_X ( \xi \omega) = \xi^2 q_X (\omega),
        \]
        where $(\ast)$ holds due to \eqref{eq:BBF_form_invariance}.
        Since $d>2$, we infer that the weight space $L_{\C,1}$ is isotropic in $L_{\C}$. Therefore,
        \[
            \Omega(Y) = \big\{ [\omega] \in \P(L_{\C,1}) \mid q_X(\omega,\overline{\omega})>0 \big\}
        \]
        is an open subset of $\P(L_{\C,1})$. It follows now from \eqref{eq:dim_grass} that
        \begin{equation}
            \label{eq:dim_Omega}
            \dim T_{[\phi(\sigma)]} \, \Omega(Y)= \dim \Hom \big( \C \sigma , H^2(X,\C)_1 / \C \sigma \big).
        \end{equation}
        
        Now, since the Hodge decomposition \eqref{eq:Hodge_decomposition_PSV} is preserved by automorphisms of $X$, we have
        \[
            H^2(X,\C)_1 = H^{2,0}(X)_1 \oplus H^{1,1}(X)_1 \oplus H^{0,2}(X)_1 ,
        \]
        where $H^{0,2}(X)_1 = 0$, because $g^{[*]}$ acts via multiplication by $\overline{\xi}$ on $H^{0,2}(X)$ and $d>2$. Hence, by \eqref{eq:dim_S} and \eqref{eq:dim_Omega} we deduce that 
        \[ 
            \dim T_0 S = h^{1,1}(X)_1 = \dim T_{[\phi(\sigma)]} \, \Omega(Y) ,
        \]
        which finishes the proof in the case $d \geq 3$.

        \medskip
        
        \emph{Step 4}: Case $d=2$.
        
        Since $d=2$, we have $\xi =\overline{\xi} = -1$, and thus $\overline{\sigma} \in H^{0,2}(X)_{-1} = H^{0,2}(X)_{1}$. The Hodge decomposition \eqref{eq:Hodge_decomposition_PSV} yields
        \begin{align*}
            H^2(X,\C)_1 & = H^{2,0}(X)_1 \oplus H^{1,1}(X)_1 \oplus H^{0,2}(X)_1 \\
            & = \C \sigma \oplus H^{1,1}(X)_1 \oplus \C \overline{\sigma} .
        \end{align*}
        By \eqref{eq:BBF_form_evaluation_1}, \eqref{eq:BBF_form_evaluation_2} and \eqref{eq:quadratic_form} we deduce that the orthogonal complement of $\sigma\in H^2(X,\C)_1$ with respect to $q_X(\cdot, \cdot)$ is the subspace $V = \C \sigma \oplus H^{1,1}(X)_1$. Therefore,
        \begin{equation}
            \label{eq:dim_orthogonal}
            \dim \Hom (\C \sigma, V/ \C \sigma) = h^{1,1}(X)_1.
        \end{equation}

        Now, since $\Omega(Y)$ is an open subset of the quadric in $\P(L_{\C,1})$ defined by the equation $q_X(\cdot) = 0 $, the tangent space $T_{[\phi(\sigma)]} \, \Omega(Y)$ is precisely the tangent space of the quadric at the point $[\phi(\sigma)]$. Since the quadratic form $q_X$ is nondegenerate, the quadric $\big\{ q(\cdot)= 0 \big\}$ is smooth. Using this, together with \eqref{eq:dim_grass}, we deduce that the tangent space of the quadric at the point $[\phi(\sigma)]$ is precisely 
        $\Hom (\C \sigma , V/\C \sigma)$. It follows now from \eqref{eq:dim_S} and \eqref{eq:dim_orthogonal} that
        \[ 
            \dim T_0 S = h^{1,1}(X)_1 = \dim T_{[\phi(\sigma)]} \, \Omega(Y) ,
        \]
        which completes the proof in the case $d = 2$.
    \end{proof}

    \subsection{A local Torelli theorem for bielliptic surfaces}
    \label{subsection:localTorelli_BS}

    In this subsection
    we describe the Kuranishi family of a bielliptic surface (see \autoref{prop:Kuranishi_family_of_BS}) and we provide subsequently a local Torelli theorem for bielliptic surfaces (see \autoref{thm:local_Torelli_bielliptic_surface}). 

    \medskip
    
    First of all, we recall some basic facts about abelian surfaces. Specifically, if $A$ is an abelian surface, then $H^2(A,\Z)$ is a free $\Z$-module of rank $ 6$ by \cite[Corollary 1.3.3]{BL04book}, while by \cite[Lemma 1.3.1]{BL04book} the cup product induces an isomorphism
    \[ 
        \bigwedge^2 H^2(A,\Z) \simeq H^4(A,\Z) \simeq \Z .
    \]
    Therefore, $H^2(A,\Z)$ has a lattice structure and the resulting nondegenerate, integral, quadratic form on $H^2(A,\Z)$ is denoted by $q_A$. Note that $\big( H^2(A,\Z) \, , q_A \big)$ is an even, unimodular, nondegenerate lattice with signature of type $(3,\ast)$, which is isomorphic to $U^{\oplus3}$, where $U$ is the hyperbolic plane over $\Z$;
    see \cite[Chapter 3, Subsections 1.3, 2.2 and 2.3]{Huy16book}.
    The direct summands of the Hodge decomposition of $H^2(A,\C)$ have dimension
    \[ 
        h^{2,0}(A) = h^{0,2}(A) = 1 \quad \text{ and } \quad h^{1,1}(A) = 4 ,
    \]
    and the symplectic form $\sigma_A$ on $A$ induces an isomorphism 
    \begin{equation}
        \label{eq:tang_cotang_iso_AS}
        \mathscr{T}_A \simeq \Omega^1_A ,
    \end{equation}
    which yields
    \[
        H^1(A,\mathscr{T}_A) \simeq H^1(A,\Omega^1_A) \simeq H^{1,1}(A) .
    \]
    
    Furthermore, the Kuranishi space $\Def(A)$ of $A$ is smooth of dimension 
    \[
        \dim \Def(A) = \dim_\C H^1(A,\mathscr{T}_A) = h^{1,1}(A) = 4 
    \]
    by the Bogomolov--Tian--Todorov unobstructedness theorem, and the Kuranishi family $\pi_A \colon \mathscr{A} \to \Def(A)$ of $A$ is universal; see, for example, \cite[Theorems 1.1 and 4.2]{Wavrik69}.
    In particular, up to shrinking $\Def(A)$, each fiber of $\pi_A$ is an abelian surface, and also the following \emph{local Torelli statement} holds: the local period map 
    \[
        \mathscr{P}_A \colon \Def(A) \to \Omega(A) , \ t \mapsto \P \big( H^{2,0}( \mathscr{A}_t ) \big),
    \]
    where 
    \[
        \Omega(A) \coloneqq \big\{ [\omega] \in \P \big( H^2(A,\C) \big) \mid q_A(\omega)=0 , \ q_A(\omega,\overline{\omega})>0 \big\}, 
    \]
    is a local isomorphism, see \cite[Subsection 10.3.2]{Voisin07book_I}.
    
    \medskip
    
    We now turn our attention to bielliptic surfaces in order to derive an analogous statement to the above. Given a bielliptic surface $B$, recall that there exist an abelian surface $A$ and a finite cyclic group $G$ of order $d \in \{ 2, 3, 4, 6 \}$ which acts freely on $A$ such that $B \simeq A/G$; see \cite[Section 6]{OS11b} and \cite[Subsections 2.1 and 2.2]{FTVB22}. The quotient map $A \to B$ is the \emph{global index-one cover} of $B$ (also referred to as the \emph{canonical cover} of $B$), see \cite[Definition 5.19]{KM98}.
    Since 
    $H^0(A,\omega_A)^G \simeq H^0(B, \omega_B) = \{ 0 \} $,
    the group $G \subseteq \Aut(A)$ must contain at least one nonsymplectic automorphism.
    In fact, by inspecting the classification table of bielliptic surfaces due to Bagnera--De Franchis, one can readily check that 
    \begin{equation}
        \label{eq:bielliptic_nonsymplectic}
        G \text{ acts on $A$ nonsymplectically} ,
    \end{equation} 
    i.e., the generator of $G$ is a purely nonsymplectic automorphism of order $d$; see \cite[Section 6]{OS11b} or \cite[Theorem, pp.\ 473-474]{Suwa69}.
    Taking \eqref{eq:tang_cotang_iso_AS} and \eqref{eq:bielliptic_nonsymplectic} into account and arguing as in the proof of \autoref{lem:low_coh_PEV}(iii), we infer that 
    \begin{equation}
        \label{eq:bielliptic_low_coh}
        H^1(B,\mathscr{T}_{B}) \simeq H^1(A,\mathscr{T}_A)^G = H^1(A, \mathscr{T}_A)_0 \simeq H^1(A,\Omega^1_A)_1 \simeq H^{1,1}(A)_1 \, .
    \end{equation}
    We can now describe the Kuranishi family of $B \simeq A/G$.
    
    \begin{prop}[The Kuranishi family of a bielliptic surface]
        \label{prop:Kuranishi_family_of_BS}
        Let $B \simeq A/G$ be a bielliptic surface, where $A$ is an abelian surface and $G \subseteq \Aut(A)$ is generated by a purely nonsymplectic automorphism $g \in G$ of order $d \in \{2, 3, 4, 6\}$ and acts freely on $A$. Let also $\pi_A \colon \mathscr{A} \to \Def(A)$ be the universal Kuranishi family of $A$ and set $\mathscr{A}_t \coloneqq \pi_A^{-1}(t)$. The following statements hold:
        \begin{enumerate}[\normalfont (i)]
            \item Let $F$ be the connected component of the fixed locus of the induced action of $G$ on $\Def(A)$ that contains the point $0\in \Def(A)$. The Kuranishi space
            $\Def(B) \simeq F$ of deformations of $B$ is smooth of dimension 
            \[ 
                h^1(B, \mathscr{T}_B) = 
                \begin{cases}
                    2 , & \text{if } \ d = 2, \\
                    1 , & \text{if } \ d \in \{3,4,6\} ,
                \end{cases}
            \]
            and the Kuranishi family 
            \[
                \pi_B \colon \mathscr{B} \coloneqq \big( \mathscr{A} \times_{\Def(A)} \Def(B) \big) / G \to \Def(B)
            \]
            of $B$ is universal for all of its fibers.
            
            \item If $V$ is a sufficiently small neighborhood of $0 \in \Def(B) \subset \Def(A)$, then for each $t \in V$ the fiber $\mathscr{B}_t$ is a bielliptic surface isomorphic to the quotient $\mathscr{A}_t/G$ of the abelian surface $\mathscr{A}_t$.
        \end{enumerate}
    \end{prop}

    \begin{proof}
        To show that $\mathscr{B} \coloneqq \big( \mathscr{A} \times_{\Def(A)} F \big) / G \to F$ is a miniversal deformation of $B \simeq \mathscr{B}_0$ and that $F$ is smooth, we closely follow the proof of \autoref{thm:locally_trivial_Kuranishi_family_of_PEV}(i), but with the following differences: in the second paragraph, we apply \autoref{thm:Doan}
        instead of both \cite[Lemma 4.6]{BL22} and \autoref{lem:extending_group_action_to_deformation} to extend the $G$-action on $A$ equivariantly to its universal family $\mathscr{A} \to \Def(A)$, while in the third paragraph, we replace \autoref{lem:low_coh_PEV}(iii) by \eqref{eq:bielliptic_low_coh}.
        
        Since, according to \cite[p.\ 107, 3rd paragraph]{Grauert74}, miniversal deformations of a given compact complex manifold are uniquely determined (up to isomorphism), we conclude that the above deformation is actually the Kuranishi family $\mathscr{B} \to \Def(B) \simeq F$ of $B$; in particular, the germ $\big( \Def(B), 0 \big)$ is reduced. Hence, we have 
        \[ 
            \dim \Def(B) = h^1(B, \mathscr{T}_B) = 
            \begin{cases}
                2 , & \text{if } \ d = 2, \\
                1 , & \text{if } \ d \in \{3,4,6\} ,
            \end{cases} 
        \]
        where the second equality holds by \cite[Theorem 4.9]{Lang79}. Next, to show that $\mathscr{B} \to \Def(B)$ is universal for $B$ (and thus for all of its fibers due to the openness of versality),
        by \cite[Theorems 1.1 and 4.2]{Wavrik69} it suffices to show that $h^0(\mathscr{B}_t, \mathscr{T}_{\mathscr{B}_t})$ is independent of $t$ in a sufficiently small neighborhood of $0 \in \Def(B)$. However, granting (ii) for a moment, this follows immediately from \cite[Theorem 4.9]{Lang79}, which shows that bielliptic surfaces carry a unique (up to scalar) holomorphic vector field.

        Consequently, to complete the proof of the theorem, it remains to show (ii). To do so, we proceed in two steps.

        \medskip

        \emph{Step 1}: 
        We will show that $G$ acts freely on any fiber of $\pi_B$ over a sufficiently small neighborhood of $0 \in \Def(B) \subset \Def(A)$.
        
        To this end, consider the nonfree locus $\mathscr{A}^\text{nf}$ of the induced $G$-action on $\mathscr{A}$, that is,
        \[ 
            \mathscr{A}^\text{nf} \coloneqq \big\{ x \in \mathscr{A} \mid G_x \neq \{ \Id_G \} \big\} ,
        \]
        which is a closed subset of $\mathscr{A}$
        and does not contain any point of $\mathscr{A}_0 \simeq A$ by assumption. Since the morphism $\pi_A \colon \mathscr{A} \to \Def(A)$ is smooth and proper, the image $\pi_A \big( \mathscr{A}^\text{nf} \big)$ is a closed subset of $\Def(A)$ that does not contain $0 \in \Def(A)$ and the restricted morphism $\pi_A|_{\mathscr{A} \setminus \mathscr{A}^\text{nf}} \colon \mathscr{A} \setminus \mathscr{A}^\text{nf} \to \Def(A) \setminus \pi \big( \mathscr{A}^\text{nf} \big)$ is also smooth and proper.
        This yields the assertion; cf.\ the proof of \cite[Proposition 1.2]{OS11b}.
        
        \medskip

        \emph{Step 2}:
        We will demonstrate that $G$ acts nonsymplectically on any fiber of $\pi_B$ over a sufficiently small neighborhood of $0 \in \Def(B) \subset \Def(A)$.

        For $t$ in a sufficiently small neighborhood of $0 \in \Def(B) \subset \Def(A)$ as in Step 1, the quotient map $\mathscr{A}_t \to \mathscr{B}_t \simeq \mathscr{A}_t/G$ is a finite \'etale cover. Therefore, $\mathscr{B}_t$ has torsion canonical class. But since Hodge numbers are constant in the smooth family $\mathscr{B} \to \Def(B)$ whose central fiber $\mathscr{B}_0$ is the given bielliptic surface $B$, we conclude that $\mathscr{B}_t \simeq \mathscr{A}_t/G$ is also a bielliptic surface. In particular, the induced $G$-action on $\mathscr{A}_t$ is nonsymplectic, as explained earlier in this subsection; see \eqref{eq:bielliptic_nonsymplectic}. This finishes the proof.
    \end{proof}

    As a short interlude, we explain here a notion of marking for bielliptic surfaces, which will be tacitly used in the proof of \autoref{thm:local_Torelli_bielliptic_surface}. Specifically, the discussion that precedes \autoref{lem:hermitian_form_PSV} remains valid also in the present framework, and consequently an identical version of \autoref{lem:hermitian_form_PSV} with respect to the Hermitian form associated with $q_A$ can be established for the canonical abelian cover $A$ of a bielliptic surface $B \simeq A/G$. This is due to the following two observations: \eqref{eq:BBF_form_invariance} can be replaced by the invariance of the cup product under homeomorphisms, 
    while \eqref{eq:BBF_form_evaluation_1} can be replaced by the equalities $[\sigma_A] \smile [\sigma_A] = 0 = [\overline{\sigma_A}] \smile [\overline{\sigma_A}]$ and the Hodge--Riemann bilinear relations; see
    \cite[Theorem 6.32]{Voisin07book_I}.
    Therefore, an \emph{$L$-marking of a bielliptic surface} can be defined in the same way as it was previously done for a primitive Enriques variety; see Subsection \ref{subsection:PEVs}.

    \medskip
    
    Finally, we establish the desired statement for bielliptic surfaces, which will only be used in the proof of \autoref{prop:deform_PEV_from_Kum}.
    
    \begin{thm}[Local Torelli theorem for bielliptic surfaces]
        \label{thm:local_Torelli_bielliptic_surface}
        Let $B \simeq A/G$ be a bielliptic surface and let 
        \[
            \Omega(B) \coloneqq \big\{ [\omega] \in \P \big( H^2(A,\C)_1 \big) \mid q_A(\omega)=0 , \ q_A(\omega,\overline{\omega})>0 \big\} 
        \]
        be the period domain for $B$ inside $\P \big( H^2(A,\C)_1 \big)$. If $\pi_B \colon \mathscr{B} \to \Def(B)$ is the universal Kuranishi family of $B$ and if $\mathscr{A}_t \coloneqq \pi_A^{-1}(t)$ is the fiber over $t \in \Def(A)$ of the universal Kuranishi family $\pi_A \colon \mathscr{A} \to \Def(A)$ of $A$, then the local period map
        \[
            \mathscr{P}_B \colon \Def(B) \to \Omega(B) , \ s \mapsto \P \big( H^{2,0}(\mathscr{A}_s) \big)
        \]
        is a local isomorphism at $0 \in \Def(B)$.
    \end{thm}

    \begin{proof}
        We may assume that $B \simeq A/G$, where $A$ is an abelian surface and $G \subseteq \Aut(A)$ is generated by a purely nonsymplectic automorphism $g \in \Aut(A)$ of order $d \in \{2, 3, 4, 6\}$ and acts freely on $A$.
        Now, arguing as in Step 1 of the proof of \autoref{thm:local_Torelli_PEV}, we obtain a holomorphic local period map 
        \[ 
            \mathscr{P}_B \colon \Def(B) \to \Omega(B) , \ s \mapsto \big[ \phi(\sigma_{\mathscr{A}_s}) \big] ,
        \]
        which is the restriction of the local period map $\mathscr{P}_A \colon \Def(A) \to \Omega(A)$ for $A$. 
        Here, $\phi$ denotes the fiberwise marking of the deformation over a simply connected base $S$ of the bielliptic surface $B \simeq A/G$ involved in the proof, and $\sigma_{\mathscr{A}_s}$ denotes the symplectic form on the fiber $\mathscr{A}_s$ over $s \in S$ of the corresponding deformation of the abelian surface $A$.
        
        Next, to conclude that $\mathscr{P}_B$ is a local isomorphism at $0 \in \Def(B)$, arguing as in Step 2 of the proof of \autoref{thm:local_Torelli_PEV}, except that we now invoke the local Torelli theorem for abelian surfaces
        and \eqref{eq:bielliptic_low_coh} instead of \cite[Proposition 5.5]{BL22} and \autoref{lem:low_coh_PEV}(iii), respectively, 
        we see that it remains to compute $\dim_\C T_{[\phi(\sigma_A)]} \, \Omega(B)$ and to verify that it is equal to $\dim_\C T_0 \Def(B)$. This will be done below by distinguishing cases for $d$ and by noting that we now work with the lattice $\big( H^2(A,\Z) \, , q_A \big)$, where the nondegenerate symmetric bilinear form $q_A(\cdot, \cdot)$ associated to $q_A(\cdot)$ is the cup product in cohomology.

        If $d \in \{3,4,6\}$, then the same argument as in Step 3 of the proof of \autoref{thm:local_Torelli_PEV} works also in our context due to the following observation: For any $\omega \in H^2(A,
        \C)_1$ we have $g^* \omega = \xi \omega$, where $\xi$ is a primitive $d$-th root of unity, and thus
        \[ 
            \omega \smile \omega \overset{(\ast \ast)}{=\!=\!=}  g^* \omega \smile g^* \omega = \xi^2 (\omega \smile \omega) ,
        \]
        where $(\ast \ast)$ holds because the cup product is preserved under homeomorphisms.
        
        If, finally, $d=2$, then we repeat verbatim the argument from Step 4 of the proof of \autoref{thm:local_Torelli_PEV}, except that we now determine the orthogonal complement of the symplectic form $\sigma_A \in H^{2,0}(A) \subseteq H^2(A,\C)_1$ with respect to $q_A(\cdot, \cdot)$ using the Hodge--Riemann bilinear relations \cite[Theorem 6.32]{Voisin07book_I}.
    \end{proof}

    \section{Applications}
    \label{section:applications}
    
    In the last section of the paper we describe how small locally trivial deformations of primitive Enriques varieties arising as finite quasi-\'etale quotients of certain punctual Hilbert schemes, generalized Kummer varieties or moduli spaces of sheaves on K3 surfaces, look like. Along the way we construct some families of such primitive Enriques varieties. The constructions presented below closely follow \cite{Beauville83a,OS11b} and generalize the ones from the latter to the singular setting, as the considered nonsymplectic group actions are free in codimension one.

    \subsection{Locally trivial deformations of $\Hilb^n(S)/G$}
    \label{subsection:deforming_PEV_from_Hilb}
    
    Consider a K3 surface $S$ with a purely nonsymplectic automorphism $g \in \Aut(S)$ of finite order $d \geq 2$ that acts freely in codimension one on $S$. Note that $S$ is projective by \cite[Proposition 6(i)]{Beauville83b} and set $G \coloneqq \langle g \rangle \subseteq \Aut(S)$. 
    Pick $n \geq 1$ and let $G^{[n]} \coloneqq \langle g^{[n]} \rangle \subseteq \Aut \big( S^{[n]} \big)$ be the induced cyclic group of automorphisms of the Hilbert scheme $S^{[n]}$ of $n$ points on $S$.
    Since $G^{[n]}$ acts on $S^{[n]}$ nonsymplectically by construction
    and freely in codimension one either by construction when $n=1$ or by \cite[Proposition 4.13]{DOTX24} when $n\geq 2$, the quotient 
    \[
        Y \coloneqq S^{[n]} / G^{[n]} = \Hilb^n(S) / G^{[n]} 
    \]
    is a primitive Enriques variety of dimension $2n \geq 2$, which is also $\Q$-factorial by \cite[Proposition 5.15]{KM98}.\footnote{Note that if the given automorphism $g \in \Aut(S)$ is an Enriques involution and if $n \geq 1$ is odd, then the induced $G^{[n]}$-action on $S^{[n]}$ is actually free by \cite[Proposition 4.1]{OS11a}, so the quotient $Y \coloneqq S^{[n]} / G^{[n]}$ is then an Enriques manifold of dimension $2n$ and index $d=2$.}
    We will show in \autoref{prop:deform_PEV_from_Hilb} that any small locally trivial deformation of $Y = S^{[n]} / G^{[n]}$ is of the same form.
    
    To this end, we first construct a family of primitive Enriques varieties with central fiber $Y$ as follows. Consider the locally trivial Kuranishi family $\mathscr{S}' \to \Deflt(S')$ of the primitive Enriques surface $S' \coloneqq S / G$. According to \autoref{thm:locally_trivial_Kuranishi_family_of_PEV}, there is a composition of morphisms 
    \[ 
        \mathscr{S} \to \mathscr{S'} \simeq \mathscr{S}/G \to \Deflt(S') 
    \]
    which gives a family of K3 surfaces with central fiber $\mathscr{S}_0 \simeq S$ and, up to shrinking $\Deflt(S')$, we may also assume that the induced action of $G$ on $\mathscr{S}$ is fiberwise nonsymplectic and free in codimension one.
    The relative Douady space of points yields then a locally trivial family
    \[
        \big( \mathscr{S} / \Deflt(S') \big)^{[n]} \to \Deflt(S')
    \]
    of IHS manifolds, which are Hilbert schemes $\mathscr{S}_t^{[n]}$ of $n$ points on the K3 surfaces $\mathscr{S}_t$,
    with central fiber $\mathscr{S}_0^{[n]} \simeq \Hilb^n (\mathscr{S}_0) \simeq S^{[n]}$, see \cite[p.\ 778]{Beauville83a}. Note that the total space 
    \[
        \big( \mathscr{S}/\Deflt(S') \big)^{[n]}
    \] 
    is naturally endowed with a $G^{[n]}$-action, which is compatible with the induced fiberwise $G^{[n]}$-action on $\mathscr{S}_t^{[n]}$. The latter is in turn nonsymplectic and free in codimension one by construction. 
    It follows from the proof of \cite[Proposition 6.2]{Graf18} that 
    \[
        \mathscr{Y}' \coloneqq \big( \mathscr{S} / \Deflt(S') \big)^{[n]} / G^{[n]} \to \Deflt(S')
    \]
    is a locally trivial deformation of $Y = S^{[n]} / G^{[n]} = \Hilb^n(S) / G^{[n]}$. Moreover, up to further shrinking $\Deflt(S')$, any member $\mathscr{Y}'_t$ of this family is a $\Q$-factorial primitive Enriques variety of the form 
    \[ 
        \mathscr{Y}'_t \simeq \mathscr{S}_t^{[n]} / G^{[n]} = \Hilb^n (\mathscr{S}_t) / G^{[n]} 
    \]
    by construction.
    
    Now, by the universality of the locally trivial Kuranishi family $\mathscr{Y}\to \Deflt(Y)$ of the primitive Enriques variety $Y$, we obtain a unique morphism $\Deflt(S') \to \Deflt(Y)$. We may also assume that these two complex manifolds are simply connected, possibly after further shrinking them.
    The next result indicates how small locally trivial deformations of $Y$ look like and generalizes \cite[Proposition 3.1]{OS11b}; see also \cite[Remark 3.2]{OS11b}.

    \begin{prop}
        \label{prop:deform_PEV_from_Hilb}
        With the same notation and assumptions as above, the morphism $\Deflt(S') \to \Deflt(Y)$ of complex manifolds is a local isomorphism at the origin $0 \in \Deflt(S')$. In particular, any small locally trivial deformation of $Y = S^{[n]} / G^{[n]}$ is of the same form.
    \end{prop}
    
    \begin{proof}
        The proof is similar to the one of \cite[Proposition 3.1]{OS11b}, but we provide the details for the reader's convenience. We identify the groups $G$ and $G^{[n]}$ with $\mu_d(\C)$ below.
        
        Let $(L,q) \coloneqq \big( H^2(S^{[n]},\Z) \, , q_{S^{[n]}} \big)$ be the Beauville--Bogomolov--Fujiki lattice of $S^{[n]}$, which is endowed with the canonical $\mu_d(\C)$-action. According to \cite[Proposition 6 and Remark on p.\ 768]{Beauville83a}, there is a $\mu_d(\C)$-equivariant primitive embedding of lattices
        \[
            \iota \colon H^2(S,\Z) \hookrightarrow H^2(S^{[n]},\Z) ,
        \]
        compatible with Hodge structures, such that 
        \[ 
            H^2(S^{[n]},\Z) \simeq \iota \left( H^2(S,\Z) \right) \oplus \Z \delta  ,
        \]
        where $2\delta=[E]$ and $E$ is the prime exceptional divisor of the $\mu_d(\C)$-equivariant Hilbert--Chow morphism
        $ S^{[n]} = \Hilb^n(S) \to S^{(n)} = \Sym^n(S) $. Since $E$ is $\mu_d(\C)$-invariant by \cite[Theorem 1]{BS12},
        we obtain an identification of weight spaces:
        \begin{equation}
            \label{eq:weight_space_identification}
            H^2(S,\C)_1 \simeq H^2(S^{[n]},\C)_1 \, .
        \end{equation}
        
        Set $L' \coloneqq \iota(H^2(S,\Z))$. Since $\iota$ is $\mu_d(\C)$-equivariant, we obtain an $L$-marking for the primitive Enriques variety $Y$ and an $L'$-marking for the primitive Enriques surface $S'$. Due to the hypothesis that $\Deflt(S')$ and $\Deflt(Y)$ are simply connected, these markings uniquely extend to markings of the considered families $\mathscr{Y} \to \Deflt(Y)$ and $\mathscr{S}' \to \Deflt(S')$, respectively.
        Taking now \autoref{dfn:period_data}(1) and \eqref{eq:weight_space_identification} into account, we observe that the period domain $\mathcal{D}_L$ for the $L$-marked primitive Enriques varieties coincides with the period domain $\mathcal{D}_{L'}$ for the $L'$-marked primitive Enriques surfaces.
        
        Finally, consider the diagram
        \begin{center}
            \begin{tikzcd}[row sep = large, column sep = large]
                \Deflt(S' )\arrow[r] \arrow[d,"\mathscr{P}'" swap] & \Deflt(Y) \arrow[d,"\mathscr{P}"] \\
                \mathcal{D}_{L'} \ar[r,equal] & \mathcal{D}_{L} ,
            \end{tikzcd}
        \end{center}
        where $\mathscr{P}'$ and $\mathscr{P}$ are the local period maps for the marked families of primitive Enriques surfaces and primitive Enriques varieties, respectively.
        By \autoref{thm:local_Torelli_PEV}, both $\mathscr{P}'$ and $\mathscr{P}$ are local isomorphisms at the reference point $0$. Thus, to complete the proof, it only remains to check that the above diagram is commutative. But this follows from the commutativity of the diagram in \cite[p.\ 780]{Beauville83a}.
    \end{proof}

    \subsection{Locally trivial deformations of $\Kum^n(A)/G$}
    \label{subsection:deforming_PEV_from_Kum}

    Consider a bielliptic surface $B$. As in Subsection \ref{subsection:localTorelli_BS}, write $B \simeq A/G$, where $A$ is an abelian surface and $G \subseteq \Aut(A)$ is generated by a purely nonsymplectic automorphism $g \in \Aut(A)$ of order $d \in \{2,3,4,6\}$ acting freely on $A$. There exists a finite \'etale cover $\widetilde{A} \to B$ factoring through $A$ such that $\widetilde{A} \simeq E \times F$, where $E$ and $F$ are elliptic curves
    and $\widetilde{T} \coloneqq \ker (\widetilde{A} \to A)$ is a finite group whose image under the second projection $\widetilde{A} \simeq E \times F \to F$ is denoted by $T$, and it also holds that $\widetilde{A} / \widetilde{T} \simeq A$. We refer to \cite[Section 6]{OS11b} and \cite[Subsections 2.1 and 2.2]{FTVB22} for more details.
    
    Next, pick $n \geq 2$ and let $G^{[n+1]} \coloneqq \langle g^{[n+1]} \rangle \subseteq \Aut \big( \Hilb^{n+1}(A) \big)$ be the induced cyclic group of automorphisms of the Hilbert scheme of $n+1$ points on $A$. 
    Note that $G^{[n+1]}$ acts on $\Hilb^{n+1}(A)$ nonsymplectically.
    According to \cite[Proposition 6.1]{OS11a}, if $d \mid n+1$ and $z \in F$ is an $(n+1)$-torsion point, then the generalized Kummer variety $\Kum^n(A)$
    is $G^{[n+1]}$-invariant. Since $G^n \coloneqq \langle g^{[n+1]} |_{\Kum^n(A)} \rangle \subseteq \Aut \big( \Kum^n(A) \big)$ acts on $\Kum^n(A)$ nonsymplectically by construction, and hence freely in codimension one by \cite[Proposition 4.13]{DOTX24}, the quotient 
    \[
        Y \coloneqq \Kum^n(A) / G^n
    \]
    is a $\Q$-factorial primitive Enriques variety of dimension $2n$.\footnote{Note that, under certain additional hypotheses, the induced $G^n$-action on $\Kum^n(A)$ is actually free by \cite[Proposition 6.4]{OS11a}, so the quotient $Y = \Kum^n(A) / G^n$ is then an Enriques manifold of dimension $2n \geq 4$ and index $d \in \{ 2,3, 4 \}$. The assumptions in question are the following ones:
    \begin{itemize}
        \item if $d = 2$, then $mz \notin T$,

        \item if $d = 3$, then $T=0$ and $mz \notin \Z \big( 1 + e^{\frac{2\pi i}{6}} \big)/3$, and

        \item if $d = 4$, then $T=0$ and $2mz \notin \Z \big(1 + e^{\frac{2\pi i}{4}} \big) / 2$,
    \end{itemize}
    where $m \coloneqq \frac{n+1}{d} \in \N^*$ and $T \subseteq F$ was defined at the beginning of this subsection.}
    
    \begin{comm}
        When $n=1$, the variety $\Kum^1(A)$ is isomorphic to the Kummer K3 surface associated to the abelian surface $A$, see \cite[p.\ 769]{Beauville83a}, and the condition $d \mid n+1$ implies that the group $G^1 \simeq \Z/2\Z$ is generated by a fixed-point-free and anti-symplectic involution of $\Kum^1(A)$. Therefore, $Y = \Kum^1(A) / G^1$ is an Enriques surface, whose small deformations are classically well understood. This is the reason why we only consider the case $n\geq 2$ in this subsection.
    \end{comm}
    
    We will demonstrate in \autoref{prop:deform_PEV_from_Kum} that
    any small locally trivial deformation of $\Kum^n(A)/G^n$ is of the same form.
    To this end, consider the Kuranishi family 
    \[
        \widetilde{\mathscr{A}} \to \Def(E \times F)
    \]
    of $\widetilde{A} \simeq E \times F$. Since the latter is endowed with a holomorphic $(G \times \widetilde{T})$-action, so is the Kuranishi family by \autoref{thm:Doan}. Let $\widetilde{\Def(B)}$ be the connected component of the fixed locus of the $(G \times \widetilde{T})$-action on $\Def(E \times F)$ containing the reference point, i.e., the one over which the fiber is isomorphic to $E \times F$, and set 
    \[
        \widetilde{\mathscr{A}}_{|\Def(B)} \coloneqq \widetilde{\mathscr{A}}\times_{\Def(E\times F)} \widetilde{\Def(B)} .
    \]
    It follows from \cite[Proposition 6.2]{Graf18}
    that the Kuranishi family of the bielliptic surface $B \simeq (E \times F) / (G \times \widetilde{T})$ is 
    \[ 
        \mathscr{B} = \widetilde{\mathscr{A}}_{|\Def(B)}/(G \times \widetilde{T}) \to \Def(B) = \widetilde{\Def(B)} .
    \]
    Observe now that the morphism $\mathscr{A} \to \mathscr{B}$ constructed in \autoref{prop:Kuranishi_family_of_BS} is the relative canonical cover of $\mathscr{B}$,\footnote{The term \enquote{relative canonical cover} is justified by the fact that, for any $t\in \Def(B)$, the restriction of $\mathscr{A}\to \mathscr{B}$ to $\mathscr{A}_t\to \mathscr{B}_t$ coincides with the canonical cover of the bielliptic surface $\mathscr{B}_t$.} that is, the index-one cover associated with any lift of the line bundle $\OO_{B_t}(K_{B_t})$ to $\mathscr{B}$. Such an extension exists by \autoref{prop:extension_of_line_bundles} and \autoref{rem:stronger_cond}. Then $\widetilde{\mathscr{A}}_{|\Def(B)} \to \Def(B)$ factors through $\mathscr{A} \to \Def(B)$, because the group $\widetilde{T} \subset G \times \widetilde{T}$ acts on $\widetilde{\mathscr{A}}_{|\Def(B)}$. Thus, 
    \[ 
        \widetilde{\mathscr{A}}_{|\Def(B)}/\widetilde{T} \cong \mathscr{A} \quad \text{and} \quad \mathscr{A}/G \cong \mathscr{B} ;
    \]
    see \cite[Section 2]{FTVB22}.
    
    At this point, after possibly shrinking $\Def(B)$, we can perform the relative Kummer construction for the family $\mathscr{A} \to \Def(B)$ of abelian varieties with central fiber $\mathscr{A}_0 \simeq A$ and obtain a family 
    \begin{equation}\label{eq:kumfamily} 
        \Kum^n \big( \mathscr{A} / \Def(B) \big) \to \Def(B)
    \end{equation}
    of IHS manifolds, which are generalized Kummer varieties $\Kum^n(\mathscr{A}_t)$ associated with the abelian surfaces $\mathscr{A}_t$, with central fiber $\Kum^n(A)$. By our choice of the $(G\times \widetilde{T})$-action on $E\times F$ it follows that the $(G\times \widetilde{T})$-action on the nearby fibers of $\widetilde{\mathscr{A}}_{|\Def(B)}\to \Def(B)$ stays of the same type.
    Then, according to \cite[Proposition 6.1]{OS11a}, the family (\ref{eq:kumfamily}) comes equipped with a $G^n$-action that is fiberwise nonsymplectic and free in codimension one. Moreover, this $G^n$-action is induced by the natural $G^{[n+1]}$-action on the relative Douady space of points
    \[ 
        \big( \mathscr{A} / \Def(B) \big)^{[n+1]}\to \Def(B).
    \]
    Thus, arguing as in Step 2 of the proof of \cite[Proposition 6.2]{Graf18}, we obtain a locally trivial deformation
    \[
        \mathscr{Y}' \coloneqq \Kum^n \big( \mathscr{A} / \Def(B) \big) / G^n \to \Def(B)
    \]
    of $Y = \Kum^n(A) / G^n$ such that any member $\mathscr{Y}'_t$ of this family is a $\Q$-factorial primitive Enriques variety of the form 
    \[ 
        \mathscr{Y}_t' \simeq \Kum^n(\mathscr{A}_t) / G^n .
    \]
    
    Now, by the universality of the locally trivial Kuranishi family $\mathscr{Y} \to \Deflt(Y)$ of $Y$, we obtain a unique morphism $\Def(B) \to \Deflt(Y)$, and we may also assume that $\Def(B)$ and $\Deflt(Y)$ are simply connected. The next result describes the small locally trivial deformations of $Y$ and generalizes \cite[Proposition 3.5]{OS11b}.
    
    \begin{prop}
        \label{prop:deform_PEV_from_Kum}
        With the same notation and assumptions as above, the morphism $\Def(B) \to \Deflt(Y)$ is a local isomorphism at the origin $0 \in \Def(B)$. In particular, any small locally trivial deformation of $Y = \Kum^n(A) / G^n$ is of the same form.
    \end{prop}

    \begin{proof}
        The argument is almost identical to the one for the proof of \autoref{prop:deform_PEV_from_Hilb}, so we omit the details and we only mention the required modifications. First, we apply \cite[Proposition 8]{Beauville83a} and the analog of \cite[Remarque, p.\ 768]{Beauville83a} for the generalized Kummer construction (see, for instance, the proof of \cite[Corollary 5(2)]{BNWS11})
        instead of \cite[Proposition 6]{Beauville83a} and \cite[Remarque, p.\ 768]{Beauville83a}, respectively, to deduce the identification of weight spaces: 
        \[ 
            H^2(A,\C)_1 = H^2 \big( \Kum^n(A), \C \big)_1 \, .
        \]
        Second, we use the analog of \cite[Proposition 10]{Beauville83a} for the Kummer construction to show the commutativity of the corresponding diagram involving local period maps. We conclude by invoking the local Torelli theorems for the primitive Enriques variety $Y = \Kum^n(A) / G^n$ (\autoref{thm:local_Torelli_PEV}) and the bielliptic surface $B = A/G$ (\autoref{thm:local_Torelli_bielliptic_surface}).
    \end{proof}

    \subsection{Locally trivial deformations of $M_S(v,H) / G$}
    \label{subsection:deforming_PEV_from_moduli}
    
    We finally examine small locally trivial deformations of primitive Enriques varieties which are obtained as quotients of moduli spaces of semistable sheaves on certain K3 surfaces.

    \medskip
    
    Consider an Enriques surface $S'$ with universal cover $S \to S'$ and recall that the covering K3 surface $S$ is endowed with a free, and hence nonsymplectic by the holomorphic Lefschetz fixed point formula,
    action of $\pi_1(S') \simeq \Z / 2\Z$, corresponding to an Enriques involution $\tau \in \Aut(S)$. Assume, in addition, that $S'$ is \emph{very general} in the sense of \cite[Proposition 5.5]{OS11a}, so that $S$ satisfies any of the equivalent conditions of \cite[Proposition 5.1]{OS11a}; e.g., $\rho(S)=10$.
    Next, consider a positive Mukai vector $v = (r,\ell,\chi-r) \in \widetilde{H}(S,\Z)$ and a $v$-generic polarization $H \in \Amp(S)$. Then the moduli space $M_S(v,H)$ is a projective irreducible symplectic variety of dimension $v^2 + 2 \geq 4$. 
    Furthermore, since $\rho(S)=10$, both $v$ 
    and $H$ are $\langle \tau \rangle$-invariant by \cite[Proposition 5.1]{OS11a}, so the group $G \coloneqq\langle \tau \rangle \simeq \Z / 2\Z$ acts on $M_S(v,H)$ as well, and this induced action corresponds to the involution $\tau^{\mathrm{ms}} \in \Aut \big( M_S(v,H) \big)$ determined by $\tau \in \Aut(S)$.
    We now claim that $\langle \tau^{\mathrm{ms}} \rangle$ acts on $M_S(v,H)$ nonsymplectically,
    and thus freely in codimension one by \cite[Proposition 4.13]{DOTX24}. Indeed, as explained in the proof of \cite[Proposition 3.3]{OS11b}, see also \cite{PeregoRapagnetta24}, the Hodge isometry \eqref{eq:Hodge_isometry_moduli} is $G$-equivariant in the present setting, so if $\sigma \in H^0(S, \Omega^2_S)$ and $\theta_\C(\sigma) \in H^0 \big( M_S(v,H), \Omega^{[2]}_{M_S(v,H)} \big) $ denote the symplectic form on $S$ and on $M_S(v,H)$, respectively, then we have 
    \[ 
        \left( \tau^{\mathrm{ms}} \right)^* \theta_\C (\sigma) = \theta_\C (\tau^* \sigma) = \theta_\C (-\sigma) =
        - \theta_\C (\sigma) ,
    \]
    as asserted.
    Therefore, the quotient 
    \[
        Y \coloneqq M_S(v,H) / G
    \]
    is a primitive Enriques variety.\footnote{Note that if, additionally, $v = (r, \ell , \chi - r)$ is primitive and $\chi$ is odd, then the induced $G$-action on the projective IHS manifold $M_S(v,H)$ is actually free by \cite[Proposition 5.2]{OS11a}, so the quotient $Y = M_S(v,H)/G$ is then an Enriques manifold of dimension $v^2 + 2$ and index $d=2$.}
    We will show in \autoref{prop:deform_PEV_from_Moduli} that any small locally trivial deformation of $Y = M_S(v,H)/G$ is of the same form.
    
    To this end, we first construct a relative moduli space $\mathscr{M}_{\mathscr{S}/\Def(S')}(\nu,\mathcal{H})\to \Def(S')$ for the family $\mathscr{S} \to \Def(S')$ of K3 surfaces with central fiber $\mathscr{S}_0 \simeq S$, which is induced by the
    Kuranishi family $\mathscr{S}' \to \Def(S')$ of the given very general Enriques surface $S' \simeq \mathscr{S}'_0$, where $\nu$ and $\mathcal{H}$ denote a relative Mukai vector and a $\nu$-generic relative polarization, respectively. 
    Recall that
    \[ 
        v = (r, \ell , \chi - r) = \big( r, c_1(L), s \big) \in \widetilde{H}(S,\Z)
    \]
    is the given Mukai vector on $S$, where $\ell \coloneqq c_1(L)$ for some holomorphic line bundle $L$ on $S$, and that we have the following maps in our context:
    \begin{center}
        \begin{tikzcd}
            S \arrow[d] \\ S'
        \end{tikzcd}
        \hspace{3em}
        \begin{tikzcd}
            \mathscr{S} \arrow[dr] \arrow[rr] && \Def(S') \\
            & \mathscr{S}' \simeq \mathscr{S} / G \arrow[ur]
        \end{tikzcd}
    \end{center}
    Denote by $L'$ (resp.\ $H'$) a holomorphic line bundle on $S'$ whose pullback to $S$ is $L$ (resp.\ $H$). 
    By \autoref{prop:extension_of_line_bundles} and \autoref{rem:stronger_cond} (resp.\ \autoref{cor:local_projectivity} and \autoref{rem:local_projectivity_family_of_EMs}) we obtain, after shrinking $\Def(S')$, a holomorphic line bundle $\mathcal{L}'$ on $\mathscr{S}'$ such that $c_1(\mathcal{L}'_0) = c_1(L')$, where $\mathcal{L}'_0 = \mathcal{L}' |_{S'}$ (resp.\ a relatively ample line bundle $\mathcal{H}'$ on $\mathscr{S}'$ over $\Def(S')$ such that $c_1(\mathcal{H}'_0) = c_1(H')$, where $\mathcal{H}'_0 = \mathcal{H}' |_{S'}$). Denote by $\mathcal{L}$ (resp.\ $\mathcal{H}$) the pullback of $\mathcal{L}'$ (resp.\ $\mathcal{H}'$) to $\mathscr{S}$. Set $\nu \coloneqq \big(r,c_1(\mathcal{L}),s\big)$ and observe that $\nu_0 \coloneqq \nu |_{S} = v$ is the given Mukai vector. Since $H = \mathcal{H} |_S \in \Amp(S)$ is $v$-generic by assumption, after possibly further shrinking $\Def(S')$, by \cite[Theorem 4.3.7]{HuyLehn10book} and \cite[Corollary 4.2 and Lemma 4.4]{PeregoRapagnetta13} we obtain a family 
    \[ 
        \mathscr{M}_{\mathscr{S}/\Def(S')}(\nu,\mathcal{H})\to \Def(S') 
    \]
    of irreducible symplectic varieties with central fiber $M_S(v,H)$, as desired. For later use, we denote by $M_{\mathscr{S}_t} (\nu_t, \mathcal{H}_t)$, $t \in S$, the members of this family, where $\mathscr{S}_t$, $\nu_t$ and $\mathcal{H}_t$ have the obvious meaning.

    Next, since $H$ and $v$ are $\langle \tau \rangle$-invariant, and since we have $H^2(\mathscr{S},\Z) \simeq H^2(\mathscr{S}_t,\Z)$ for any $t\in \Def(S')$ due to the local triviality of the family $\mathscr{S} \to \Def(S')$, we infer that $\mathcal{H}$ and $\nu$ are also invariant with respect to the naturally induced $G$-action on $\mathscr{M}_{\mathscr{S}/\Def(S')}(\nu,\mathcal{H})$.
    Hence, $G$ also acts in a fiberwise way on the family $\mathscr{M}_{\mathscr{S}/\Def(S')}(\nu,\mathcal{H}) \to \Def(S')$ and by construction this fiberwise action is nonsymplectic and free in codimension one.\footnote{As mentioned in \cite[p.\ 1642]{OS11b}, when $v$ is primitive and $\chi$ is odd, the induced $G$-action on the projective IHS manifold $M_{S_t} (v_t, H_t)$ is again free, and hence nonsymplectic by the holomorphic Lefschetz fixed point formula.}
    Since the family 
    \[
        \mathscr{M}_{\mathscr{S}/\Def(S')}(\nu,\mathcal{H}) \to \Def(S')
    \]
    is locally trivial by \cite[Lemma 2.21(1)]{PeregoRapagnetta23}, arguing as in Step 2 of the proof of \cite[Proposition 6.2]{Graf18} we deduce that 
    \[  
        \mathscr{Y}' \coloneqq \mathscr{M}_{\mathscr{S}/\Def(S')}(\nu,\mathcal{H})/G \to \Def(S') 
    \]
    is a locally trivial deformation of $Y = M_S(v,H) / G$ such that any member $\mathscr{Y}'_t$ of this family is by construction a primitive Enriques variety of the form 
    \[ 
        \mathscr{Y}_t' \simeq M_{\mathscr{S}_t} (\nu_t, \mathcal{H}_t) / G .
    \]
    
    Now, by the universality of the locally trivial Kuranishi family $\mathscr{Y} \to \Def(Y)$ of $Y = M_S(v,H) / G$, we obtain a unique morphism $\Def(S') \to \Deflt(Y)$ of complex manifolds, and we may also assume that these two spaces are simply connected. The next result indicates how small locally trivial deformations of $Y$ look like and constitutes a vast generalization of \cite[Proposition 3.3]{OS11b}.

    \begin{prop}
        \label{prop:deform_PEV_from_Moduli}
        With the same notation and assumptions as above, the morphism $\Def(S') \to \Deflt(Y)$ of complex manifolds is a local isomorphism at the origin $0 \in \Def(S')$. In particular, any small locally trivial deformation of $Y = M_S(v,H) / G$ is of the same form.
    \end{prop}
    
    \begin{proof}
        The proof of the above assertion is essentially the same as the one of \cite[Proposition 3.3]{OS11b}, but we provide the details for the benefit of the reader.
        
        First, we claim that 
        \begin{equation}
            \label{eq:weight_space_in_orthogonal_complement}
            H^2(S,\C)_1 \subseteq (v^\perp)_\C .
        \end{equation}
        Indeed, since the considered Mukai vector $v = (r, \ell, \chi - r) \in \widetilde{H}(S,\Z)$ is $G$-invariant by our assumption that $\rho(S)=10$ and since the BBF lattice $\big( H^2(S,\Z) , q_S \big)$ of $S$ is endowed in the present setting with an orthogonal $G$-action, for any $\omega \in H^2(S,\C)_1$, which is identified with the Mukai vector $ w \coloneqq (0, \omega,0) \in \widetilde{H}(S,\C)$, 
        we have 
        \[ 
            q_S(\ell, \omega) = q_S (\tau^* \ell, \tau^*\omega) = q_S (\ell, -\omega) = - q_S(\ell, \omega) , 
        \]
        whence
        \[
            \langle v,w \rangle_{\C} = q_S(\ell, \omega) = 0 ,
        \]
        which proves the claim.
    
        Next, since the Hodge isometry \eqref{eq:Hodge_isometry_moduli} is $G$-equivariant
        and due to \eqref{eq:weight_space_in_orthogonal_complement}, we obtain an identification of weight spaces:
        \[ 
            H^2(S,\C)_1 \simeq H^2 \big( M_S(v,H),\C \big)_1 \, . 
        \]
        In view of the above, we also readily deduce that the period domain $\Omega(S')$ for the Enriques surface $S'$ coincides with the period domain $\Omega(Y)$ for the primitive Enriques variety $Y$. 
        
        In conclusion, to prove the statement, we argue exactly as in the third paragraph of the proof of \autoref{prop:deform_PEV_from_Hilb}, observing that the corresponding diagram involving local period maps is commutative by the very definition of the period maps and the fact that the symplectic form on $M_S(v,H)$ arises from the symplectic form on $S$ by \eqref{eq:Hodge_isometry_moduli}.
    \end{proof}

    We conclude the paper with the following two observations.
    
    \begin{rem}~
        \begin{enumerate}[(1)]
            \item Note that \cite[Remark 3.4]{OS11b} remains valid in our more general setting; namely, there exist primitive Enriques varieties of the form $M_S(v,H) / G$, where $G \simeq \Z / 2\Z$, which arise from Enriques surfaces that are more special than the very general ones considered above.

            \item 
            In the above setting, where we work with the universal covering K3 surface $S$ of a very general Enriques surface $S'$, consider a relative positive Mukai vector of the form 
            \[ 
                \nu = \big( 1, c_1(\mathcal{L}), s \big) ,
            \]
            where $\mathcal{L}$ is a holomorphic line bundle on $\mathscr{S}$ arising from one on $\mathscr{S}'$ 
            and $\mathcal{H}$ is a relative polarization on $\mathscr{S}$ stemming from one on $\mathscr{S}'$. 
            Recall that the group $G \simeq \Z / 2\Z$ acts on the induced family $\mathscr{M}_{\mathscr{S}/\Def(S')}(\nu,\mathcal{H})\to \Def(S')$ in a fiber-preserving way.
            Up to shrinking $\Def(S')$, the relative moduli space
            \[
                \mathscr{M}_{\mathscr{S}/\Def(S')}(\nu,\mathcal{H}) \to \Def(S')
            \]
            coincides with the relative Douady space 
            \[
                \big( \mathscr{S} / \Def(S') \big)^{[c_2]} \to \Def(S')
            \]
            of $c_2$ points of the family $\mathscr{S}\to \Def(S')$ of K3 surfaces, see \cite[Chapter 3, p.\ 203, Par.\ 3.3]{Huy03book}, where $c_2$ denotes the second Chern class of the rank-one coherent sheaves belonging to the relative moduli space.
            Thus, 
            as both $\nu$ and $\mathcal{H}$ are fiberwisely preserved by construction, 
            the induced family
            \[ 
                \mathscr{M}_{\mathscr{S}/\Def(S')}(\nu,\mathcal{H})/G \to \Def(S')
            \]
            recovers the example studied in Subsection \ref{subsection:deforming_PEV_from_Hilb}.
        \end{enumerate}
    \end{rem}

    \bibliographystyle{amsalpha}
	\bibliography{BibliographyForPapers}

\end{document}